\documentclass{article}

\usepackage{a4wide}
\usepackage{cite}
\usepackage[utf8]{inputenc}
\usepackage[T1]{fontenc}

\usepackage[colorlinks=true, pdfstartview=FitV, linkcolor=blue,citecolor=blue, urlcolor=blue]{hyperref}
\usepackage[french,english]{babel}
\usepackage{amsmath}
\usepackage{amscd}
\usepackage{amssymb}
\usepackage{amsrefs}
\usepackage{amsthm}
\usepackage{cases}
\usepackage{framed}
\usepackage{graphicx}
\usepackage{latexsym}
\usepackage{color}
\usepackage{array,multirow,makecell}
\usepackage[capitalise]{cleveref}
\usepackage{comment}
\usepackage{enumerate}

\usepackage{bbm}

\definecolor{purple}{RGB}{100, 0, 200}

\newcommand{\px}{\partial_x}
\newcommand{\eq}[1]{\begin{equation}
\begin{split}
#1
\end{split}
\end{equation}}
\newcommand{\eqh}[1]{\begin{equation*}
\begin{split}
#1
\end{split}
\end{equation*}}
\newcommand{\lr}[1]{\left( #1 \right)}
\newcommand{\pt}{\partial_t}
\newcommand{\re}{\rho_\ep}
\newcommand{\ue}{u_\ep}
\newcommand{\we}{w_\ep}
\newcommand{\pe}{p_\ep}
\newcommand{\vpe}{\varphi_\ep}
\newcommand{\T}{\mathbb{T}}

\setcellgapes{1pt}
\makegapedcells

\theoremstyle{plain} 
\newtheorem{thm}{Theorem}[section]
\newtheorem{prop}[thm]{Proposition}
\newtheorem{lem}[thm]{Lemma} 
\newtheorem{cor}[thm]{Corollary} 

\theoremstyle{definition}

\theoremstyle{remark} 
\newtheorem{rmk}[thm]{Remark}

\DeclareMathOperator*{\esssup}{ess\,sup}

\newcommand{\R}{\mathbb{R}}

\newcommand{\ep}{\varepsilon}

\newcommand{\tV}{\widetilde{V}}
\newcommand{\urho}{\overline{\rho}}
\newcommand{\lrho}{\underline{\rho}}
\newcommand{\red}[1]{\textcolor{red}{#1}}

\title{Hard congestion limit of the dissipative Aw-Rascle system}
\author{N. Chaudhuri\footnote{Imperial College London, London, United Kingdom; n.chaudhuri@imperial.ac.uk},\;
L. Navoret\footnote{University of Strasbourg, Strasbourg Cedex, France; laurent.navoret@math.unistra.fr},\;
C. Perrin\footnote{Aix Marseille Univ, CNRS, I2M, Marseille, France; charlotte.perrin@cnrs.fr},\; E. Zatorska\footnote{Imperial College London, London, United Kingdom; e.zatorska@imperial.ac.uk}
}

\begin{document}

\maketitle

\begin{small}
\begin{abstract}
In this study, we analyse the famous Aw-Rascle system in which  the difference between the actual and the desired velocities (the offset function) is a gradient of a singular function of the density. This leads to a dissipation in the momentum equation which vanishes when the density is zero. The resulting system of PDEs can be used to model  traffic or suspension flows in one dimension with the maximal packing constraint taken into account.
After proving the global existence of smooth solutions, we study the so-called ``\emph{hard congestion limit}'', and show the convergence of a subsequence of solutions towards a weak solution of an hybrid free-congested system.
In the context of suspension flows, this limit can be seen as the transition from a suspension regime, driven by lubrication forces, towards a granular regime, driven by the contact between the grains.
\end{abstract}

	
\end{small}

\section{Introduction}
The purpose of this work is to study a singular limit $\ep\to 0$ for the following generalization of the Aw-Rascle \cite{AR} and Zhang \cite{Zhang} system 
\begin{subnumcases}{\label{eq:AW-0}}
\partial_t \rho_\ep + \partial_x(\rho_\ep u_\ep) = 0, \label{eq:AW-0_cont}\\
\partial_t(\rho_\ep w_\ep) + \partial_x(\rho_\ep u_\ep w_\ep) = 0,\label{eq:AW-0_mom}
\end{subnumcases}
on one-dimensional periodic domain $\Omega=\T$. The unknowns of the system are the density $\re$ and the velocity of motion $u_\ep$. The quantity $w_\ep$ denotes the desired velocity of motion and it differs from the actual velocity $u_\ep$ by the offset function. This function describes the cost of moving in certain direction and it depends on the congestion of the flow.  In our case the offset function is equal to the gradient of $\pe=\pe(\re)$, more precisely:
\eq{\label{def:w}
\we=\ue+\px \pe(\re),}
where
\eq{\label{offset}
p_\ep(\rho_\ep) = \ep\dfrac{\rho_\ep^\gamma}{(1-\rho_\ep)^\beta},\quad \gamma \geq 0, \quad \beta > 1.}
This singular function plays formally the role of a barrier by preventing the density to exceed the maximal fixed threshold $\bar \rho \equiv 1$.
The motivation to study this model and it's asymptotic limit $\ep\to 0$ comes mainly from two areas of applications:

{\bf{The Aw-Rascle model for traffic \cite{AR}.}}  The system \eqref{eq:AW-0} with scalar offset function, i.e. with $w_\ep=w=u +\rho^\gamma$ for $\gamma>1$ has been derived from the Follow the Leader (FTL) microscopic model of one lane vehicular traffic  in \cite{AwKlar}. The drawback of that model is that  the offset function $\rho^\gamma$, does not preserve the maximal density constraint, i.e. solutions satisfy the maximal density constraint $\rho^0\leq \bar\rho$ initially but evolve in finite time to a state which violates this constraint.  Moreover, the velocity offset should be very small unless the density $\rho$ is very close to the maximal value, $\bar\rho=1$. 
Indeed, the drivers do not reduce their speed significantly if the traffic is not congested enough. 
To incorporate these features  the authors of \cite{BDDR} proposed to work with the asymptotic limit ($\ep\to0$) of \eqref{eq:AW-0} with $w_\ep=u_\ep+p_\ep(\re)$, and a singular scalar offset function $p_\ep$ given by \eqref{offset}. The singular Aw-Rascle system and its asymptotic limit $\ep \to 0$ has been studied numerically in~\cite{berthelin2017}, and derived from a FTL approximation in~\cite{berthelin2017FTL}.
To be able to use this model in the multi-dimensional setting, where velocity and offset function should have the same physical dimension, a possible way would be to take for the offset a gradient rather than a scalar function (nevertheless, see the recent paper~\cite{berthelin2022} for a proposition and analysis of a multi-d extension of the classical Aw-Rascle model). 
The use of a gradient can be interpreted as ability of the driver to relax their velocity to an average of the speed of the front and the rear vehicles, weighted according to the local density. So, unlike in the classical Aw-Rascle model, both front and rear interactions would have to be incorporated at the level of the particle model. This seems to be a reasonable assumption for interactions between vehicles that can change lanes and overtake each others.

{\bf{The lubrication model.}} Equations (\ref{eq:AW-0}-\ref{offset}) appear also in modeling of suspension flows, i.e. flows of grains suspended in a viscous fluid. To explain this context better, note that system \eqref{eq:AW-0} with $w_\ep$ given by \eqref{def:w} can be rewritten (formally) as the pressureless compressible Navier-Stokes equations with density dependent viscosity coefficient 
\begin{subnumcases}{\label{eq:AW-1}}
\partial_t \rho_\ep + \partial_x(\rho_\ep u_\ep) = 0, {\label{eq:AW-1_cont}}\\
\partial_t(\rho_\ep u_\ep) + \partial_x(\rho_\ep u_\ep^2) - \partial_x\left( \lambda_\ep(\re) \partial_x u_\ep\right) = 0,{\label{eq:AW-1_mom}}
\end{subnumcases}
where
\eq{\label{lambda}
\lambda_\ep(\re)=\re^2 p'_\ep(\rho).}
In system \eqref{eq:AW-1} the singular diffusion coefficient $\lambda_\ep(\rho_\ep)$ represents the repulsive lubrication forces and $\ep$ is linked to the viscosity of the interstitial fluid. 
The previous system has been rigorously derived from a microscopic approximation in~\cite{LM}.
The limit $\ep \to 0$ models the transition between the suspension regime, dominated by the lubrication forces, towards the granular regime dictated by the contacts between the solid grains.

\bigskip

Formally, performing the limit  $\ep\to0$ in \eqref{eq:AW-0} (or equivalently in \eqref{eq:AW-1}), we expect to get the solution $(\rho,u)$ of the  compressible  pressureless Euler system, at least when $\rho<1$. 
In the region where $\rho=0$ we expect that the singularity of the offset function  \eqref{offset} (equivalently the singularity of the viscosity coefficient \eqref{lambda}) will prevail giving rise to additional  forcing term. The limiting equations then read
\begin{subnumcases}{\label{eq:limit}}
\partial_t \rho + \partial_x (\rho u) = 0 \label{eq:limit-cont}\\
\partial_t (\rho u + \partial_x \pi) + \partial_x \big( (\rho u + \partial_x \pi) u\big) = 0 \label{eq:limit-mom} \\
0 \leq \rho \leq 1, \quad (1-\rho)\pi = 0, \quad \pi \geq 0,
\end{subnumcases}
where  $\pi$ is the additional unknown obtained as a limit of certain singular function of $\re$, that will be specified later on.  This limiting system has been derived formally before in the papers of Lefebvre-Lepot and Maury \cite{LM} and then  the Lagrangian solutions based on an explicit formula using the monotone rearrangement associated to the density were constructed by Perrin and Westdickenberg~\cite{PW}.    
As explained in this latter work, the previous system can be related to the constrained Euler equations, studied for instance by Berthelin in~\cite{berthelin2002} or Preux and Maury~\cite{maury2017}, by splitting the momentum equation~\eqref{eq:limit-mom} as follows:
\begin{equation}\label{eq:Euler-FC}
\begin{cases}
\partial_t (\rho u) + \partial_x(\rho u^2) + \partial_x P = 0, \\
\partial_t \pi + u \partial_x \pi = - P.
\end{cases}
\end{equation}
In~\cite{berthelin2002}, the constrained Euler equations are obtained through the \emph{sticky blocks} approximation, while in~\cite{degond2011} and \cite{BP} the system is approximated by the compressible Euler equations with singular pressure $P_\ep = P_\ep(\rho_\ep)$.

Similar asymptotic limit passage $\ep\to0$ was analysed in the multi-dimensional setting by Bresch, Necasova and Perrin \cite{BNP} in the case of heterogeneous fluids flows described by compressible Birkmann equations with singular pressure and bulk viscosity coefficient. 
The full compressible Navier-Stokes with exponentially singular viscosity coefficients and pressure was considered by Perrin in \cite{P2016,P2017}. 
The asymptotic limit when $\ep\to0$ in the singular pressure term  that leads to the two-phase compressible/incompressible Navier-Stokes equations was considered even earlier in the context of crowd dynamics, see Bresch, Perrin and Zatorska \cite{BPZ}, Perrin and Zatorska \cite{PZ}, Degond, Minakowski and Zatorska \cite{DMZ}, Degond et al. \cite{DMNZ}. Moreover, an interested reader can also consult Lions and Masmoudi~\cite{PlLM} and Vauchelet and Zatorska \cite{VZ} for different approximation of the two-phase system.
For an overview of results and discussion of models described by the free/congested two-phase flows we refer to \cite{perrin2018}.

\bigskip
Our paper contains two main results related respectively to the existence of solutions to~\eqref{eq:AW-0} at $\ep$ fixed, and the convergence as $\ep \to 0$ of the solutions towards a solution of the limit system~\eqref{eq:limit}. 
Let us be a little bit more precise about the framework and the difficulties associated to system~\eqref{eq:AW-0}.

To study~\eqref{eq:AW-0} at $\ep$ fixed, we take advantage of the reformulation of the system as (pressureless) Navier-Stokes equations with a density dependent viscosity, namely Eq.~\eqref{eq:AW-1}.
It is now well-known that, in addition to the classical energy estimate which provides a control of $\sqrt{\lambda_\ep} \partial_x u_\ep$, the BD entropy (BD for Bresch and Desjardins~\cite{bresch2006}) yields a control of the gradient of a function of the density $\rho_\ep$.
We will show in this paper, that this estimate is precisely the key ingredient to ensure the maximal density constraint, namely we will show that $\|\rho_\ep\|_{L^\infty_{t,x}} \leq C_\ep$ for some constant $C_\ep < 1$ tending to $1$ as $\ep \to 0$. \\
As $\ep \to 0$, a main issue which is common to the analysis of Navier-Stokes equations with (degenerate close to vacuum) density dependent viscosities, is the fact that, a-priori, we do not have any uniform control in $L^p$ of $\partial_x u_\ep$. 
Therefore, the identification of the limit of the nonlinear convective term $\rho_\ep u_\ep^2$ is not direct.
An important difference with~\cite{P2017} and studies on Navier-Stokes equations with degenerate viscosities, is that we have here a vanishing viscosity, namely the viscosity $\lambda_\ep(\rho)$ goes to $0$ as $\ep \to 0$ for any $\rho < 1$.  
As a result, the control the gradient of $\rho_\ep$ provided by the BD entropy is not uniform with respect to $\ep$. 
This prevents us to use Mellet-Vasseur~\cite{MV} type of estimates to pass to the limit the convective term.
An alternative point of view is provided by the work of Boudin~\cite{boudin2000}. 
It concerns the vanishing viscosity limit for pressureless gases, namely it is the study of Eq.~\eqref{eq:AW-1} where the singular viscosity term $\partial_x(\lambda_\ep \partial_xu_\ep)$ is replaced by the non-singular term $\ep \partial_x^2u_\ep$.
The key ingredient of~\cite{boudin2000} is the use of the concept of \emph{duality solutions} for the limit pressureless gas equations introduced by Bouchut and James in~\cite{bouchut1998,bouchut1999}.  
In this framework, it is particularly important to ensure a \emph{one-sided Lipschitz condition}, or Oleinik entropy condition, on the velocity field. 
It is related to the \emph{compressive property} of the dynamics which turnes out to be useful also in other ``compressive systems'' such as aggregation equations (see for instance the works of James and Vauchelet~\cite{james2013,james2016}).
Note that Berthelin in~\cite{berthelin2002} precisely derived such estimate, $\partial_x u \leq 1/t$, for solutions of the constrained Euler equations obtained through the sticky blocks approximation.
Building upon the recent developments of Constantin et al.~\cite{constantin2020} (see also~\cite{BH}) around the regular solutions for the Navier-Stokes equations, we  derive the $\ep$-uniform one-sided Lipschitz condition $\partial_x u_\ep \leq 1/t$ on the approximate solution. This estimate requires no vaccum at the level of fixed $\ep$.
A ($\ep$ dependent) lower bound on the density $\rho_\ep$ is derived by the use of additional artificial $\ep$-dependent viscosity in $\lambda_\ep(\rho)$. 
This artificial viscosity will be shown to converge to $0$ as $\ep \to 0$ and, therefore, will not perturb the limit system.

\bigskip

The outline of this paper is to first show that for $\ep$ fixed, $(\re,u_\ep)$ is a regular solution to \eqref{eq:dual_ep}.
This step will be done not for the actual system~\eqref{eq:dual_ep}, but for certain approximation of system~\eqref{eq:AW-1} described in Section~\ref{Sec:Approx}. 
We prove the existence of regular solutions to this system in Section~\ref{Sec:3} following the approach of Constantin et al.~\cite{constantin2020}.
Then, we will derive estimates uniform with respect to $\ep$, including the Oleinik entropy condition, in Section \ref{Sec:4}. Finally, we justify the limit passage $\ep\to0$ and conclude the proof of our second main result in Section \ref{Sec:5}. 
For the convenience of the reader we included in the Apendix all details of technical estimates of higher regularity of solutions for $\ep$ fixed.

\section{Approximation and main results}\label{Sec:Approx}
Our first main result concerns the existence of strong solutions to the approximation of system~\eqref{eq:AW-0} involving artificial viscosity corresponding to function $\varphi_\ep(\re)$ introduced below. We will show in the course of the proof of our second main result that the approximate term converges to zero strongly in the limit $\ep\to 0$.
With a slight abuse of notation, we re-define the functions $w_\ep$, and  $\lambda_\ep(\re),$ from the introduction including this approximation.

We consider system \eqref{eq:AW-0} with $w_\ep$ given by 
\begin{equation}\label{w_ep}
w_\ep = u_\ep + \partial_x p_\ep(\rho_\ep) + \underbrace{\partial_x \varphi_\ep(\rho_\ep)}_{\text{approximation}}, 
\end{equation}
where
\eq{\label{def:pvar}
p_\ep(\rho_\ep) = \ep\dfrac{\rho_\ep^\gamma}{(1-\rho_\ep)^\beta}, \quad 
\varphi_\ep(\rho_\ep)  = \dfrac{\ep}{\alpha-1} \rho^{\alpha-1}, \quad
\gamma \geq 0, \quad \beta > 1, \quad \alpha \in \lr{0,\frac{1}{2}}.
}
The approximation of \eqref{eq:AW-1} is thus equal to
\begin{subnumcases}{\label{eq:AW-2}}
\partial_t \rho_\ep + \partial_x(\rho_\ep u_\ep) = 0, \label{eq:AW-2-cont}\\
\partial_t(\rho_\ep u_\ep) + \partial_x(\rho_\ep u_\ep^2) - \partial_x\left(\lambda_\ep(\rho_\ep) \partial_x u_\ep\right) = 0,{\label{eq:AW-2_mom}}
\end{subnumcases}
with $\lambda_\ep$ re-defined as
\eqh{\lambda_\ep(\rho_\ep) = \rho_\ep^2 p'_\ep(\rho_\ep) +  \underbrace{\rho_\ep^2 \varphi'_\ep(\rho_\ep)}_{\text{approximation}}.}
We supplement this system with the following set of initial conditions
\eq{\label{initial}
\rho_\ep|_{t=0}=\rho^0_{\ep},\quad u_{\ep}|_{t=0}=u^0_{\ep}.
}
The existence of unique global smooth solution to the approximate problem \eqref{eq:AW-2} at $\ep>0$ fixed is stated in the theorem below.
\begin{thm}\label{thm:main1}
Let $\ep>0$ be fixed, and let $p_\ep,\ \varphi_\ep$ be given by \eqref{def:pvar}. Assume that the initial data \eqref{initial} satisfy $\rho^0_{\ep}, u^0_{\ep}\in H^3(\mathbb{T})$, with $0 < \rho^0_{\ep} < 1$.

Then, for all $T >0$, there exists a unique global solution $(\rho_\ep, u_\ep)$ to system \eqref{eq:AW-2} such that $0 < \rho_\ep(t,x) < 1$ for all $t \in [0, T]$, $x \in \mathbb{T}$, and 
\begin{equation}\label{eq:regularity}
 \rho_\ep \in \mathcal{C}([0,T]; H^3(\mathbb{T})), \qquad u_\ep \in  \mathcal{C}([0,T]; H^3(\mathbb{T})) \cap L^2(0,T; H^{4}(\mathbb{T})). 
\end{equation} 
% \red{Ewe: Add uniform estimates to the first theorem or to the limit passage?}
\end{thm}

\bigskip
We further show that $\re,\ue$ from the class \eqref{eq:regularity} satisfy some uniform in $\ep$ estimates that allow us to justify the asymptotic limit $\ep\to0$ in the weak sense.
To this end, we rewrite system \eqref{eq:AW-2} as follows
\begin{subnumcases}{\label{eq:dual_ep}}
\partial_t \rho_\ep + \partial_x (\rho_\ep u_\ep) = 0, \label{eq:dual_ep-cont} \\
\partial_t \lr{\rho_\ep u_\ep + \partial_x \pi_\ep(\re)} + \partial_x \big(\lr{\rho_\ep u_\ep + \partial_x \pi_\ep(\re)} u_\ep \big) = 0, \label{eq:dual_ep-mom}
\end{subnumcases}
where we denoted
\eq{\label{def:pi}
\pi_\ep'(\re)=\re \pe'(\re)+\re\varphi'_\ep(\re).}
We show that for $\ep\to0$ the solutions of \eqref{eq:dual_ep} converge to an entropic weak solution of \eqref{eq:limit}
with the unknowns $\rho,\ u,\ \pi$.
We have the following result.
\begin{thm}\label{thm:main2}
Let assumptions from Theorem \ref{thm:main1} be satisfied, and moreover let
\begin{align}
&0 < \rho^0_\ep(x) \leq 1- C_0 \ep^{\frac{1}{\beta-1}} \quad \forall \ x \in \mathbb{T}, \label{hyp:rho}\\
&\|\sqrt{\rho^0_\ep} u^0_\ep\|_{L^2_x} +  \esssup {(\lambda_\ep(\rho_\ep^0)\partial_x u_\ep^0)} 
\leq C, \label{hyp:vel}\\
&\|\partial_x \pi_\ep(\rho^0_\ep)\|_{L^2_x}\leq C,\\
&0 < \underline{M}^0 \leq M^0_\ep = \int_\T \rho^0_{\ep}\,dx \leq \overline{M}^0 < |\T|,\label{hyp:mass}
\end{align}
for some $C_0, C,\underline{M}^0,\overline{M}^0 > 0$ independent of $\ep$.
Then:
\begin{enumerate}
    \item The solution $(\rho_\ep, u_\ep)$ given by Theorem~\ref{thm:main1} satisfies the following uniform estimates
    \begin{align}
   \rho_\ep(t,x)\leq 1 - C \ep^{\frac{1}{\beta-1}} \quad \forall  \ (t,x) \in [0,T] \times \T,\qquad  
   \|\pi_\ep\|_{L^\infty_t H^1_x}
   %+ \|u_\ep\|_{L^\infty_{t,x}} 
   \leq C,
    \end{align}
    for some $C >0$, independent of $\ep$, and the one-sided Lipschitz condition
    \begin{equation}
    \partial_x u_\ep(t,x) \leq \frac{1}{t} \quad \forall \ (t,x) \in \ ]0,T ] \times \mathbb{T}.
    \end{equation}
    Moreover, the following inequality holds for any $S \in \mathcal{C}^1(\R)$, convex:
    \[
    \partial_t (\rho_\ep S(u_\ep)) + \partial_x(\rho_\ep u_\ep S(u_\ep)) - \partial_x(S'(u_\ep) \lambda_\ep(\rho_\ep) \partial_x u_\ep) \leq 0 ,
    \quad \forall \ (t,x) \in (0,T) \times \T.
    \]

    \item Let in addition
    \eq{
    &\rho_\ep^0\to \rho^0 \quad weakly\ in \ L^2(\mathbb{T}),\\
    &\rho_\ep^0 u_\ep^0\to \rho^0 u^0 \quad weakly\ in \ L^2(\mathbb{T}),\\
    &\px\pi_\ep(\rho_\ep^0)\to \px \pi^0 \quad weakly\ in \ L^2(\mathbb{T}). 
    }
    Then there exists a subsequence  $(\rho_\ep,u_\ep, \pi_\ep(\rho_\ep))$ of solutions to \eqref{eq:dual_ep}-\eqref{def:pi} with initial data $(\re^0,u_\ep^0,\pi_\ep(\rho^0_\ep))$, which converges to $(\rho,u,\pi)$ a weak solution of \eqref{eq:limit} with initial datum $(\rho^0,u^0,\pi^0)$.\\
    More precisely we have $0 \leq \rho \leq 1$ a.e. and the following convergences hold:
    \begin{align*}
    &\rho_\ep \rightharpoonup \rho \quad \text{weakly-* in} \quad L^\infty((0,T) \times \T), \\
    &\pi_\ep(\rho_\ep) \rightarrow \pi \quad \text{weakly-* in} \quad L^\infty (0,T; H^1(\T)), \\
    &u_\ep \rightarrow u \quad \text{weakly-* in}  \quad   L^\infty((0,T) \times \T).
\end{align*}
Eventually, the following entropy conditions hold:
\begin{itemize}
    \item Oleinik condition 
    \begin{equation}
    \partial_x u \leq \dfrac{1}{t} \qquad \text{in}\quad \mathcal{D}',
    \end{equation}
    \item entropy inequality:
    \begin{equation}
    \partial_t (\rho S(u)) + \partial_x(\rho u S(u)) + \partial_x \Lambda_S \leq 0 \qquad \text{in}\quad \mathcal{D}',
    \end{equation}
    for convex $S \in \mathcal{C}^1(\R)$, and $\Lambda_S \in \mathcal{M}_{t,x}$ satisfying $|\Lambda_S| \leq \mathrm{Lip}_S |\Lambda|$ where $\Lambda = \overline{\lambda_\ep(\rho) \partial_x u}\in \mathcal{M}_{t,x}$.
\end{itemize}
\end{enumerate}
\end{thm}

\bigskip
\begin{rmk}

The hypotheses on initial data are significantly stronger than the usual energy-related bounds. In particular:
\begin{itemize}
    \item  The assumptions for the initial velocity include the upper bound for ${(\lambda_\ep(\rho_\ep^0)\partial_x u_\ep^0)}$ uniformly in $\ep$. This is to deduce the Oleinik entropy condition discussed in the introduction.
    \item The initial condition~\eqref{hyp:mass} amounts to say the limit system (for $\ep=0$) cannot be fully congested. 
    This condition is required to control the singular potential $\pi_\ep(\rho_\ep)$ in Section~\ref{Sec:estim-ep}. Analogous constraint has been imposed to study the asymptotic limit of the  compressible Navier-Stokes equations with singular pressure, see \cite{PZ, PWZ}.
\end{itemize}

\end{rmk}

\begin{rmk}
The choice of the approximate function $\varphi_\ep$ is motivated by the uniform lower bound of the density. It was shown in \cite{MV} that, for viscosity coefficient proportional to $\rho^\alpha$ with $\alpha\in [0,1/2)$, the weak solutions to compressible Navier-Stokes have density bounded away from zero by a constant. We use this property at the level of $\ep$ being fixed in order to derive the Oleinik entropy condition, but we also show that in the limit passage $\ep\to0$, $\varphi_\ep$ converges to zero strongly.
\end{rmk}

\begin{rmk}
The main difficulty in studying the $\ep\to 0$ limit passage is to justify convergence of the nonlinear terms. In particular, to pass to the limit in the convective term $\rho_\ep u_\ep$ one needs compactness of the velocity sequence with respect to space-variable. For compressible Navier-Stokes equations with constant viscosity coefficient this sort of information is deduced directly from the a-priori estimates. When the viscosity coefficient is degenerate, one can compensate lack of the a-priori estimate by higher regularity of the density via the so-called Bresch-Desjardins estimate.
Here, the compactness w.r.t. space follows directly from the one-sided Lipschitz condition, which is possible to deduce because the system is pressureless. 
\end{rmk}

\begin{rmk}[More general singular functions $p_\ep$]
The specific form of the function $p_\ep$ (which blows up close to 1 like a power law) is used in the paper to exhibit the small scales associated to the singular limit $\ep \to 0$ (see in particular Lemma~\ref{lem:up_rho}).
Nevertheless, we expect similar results for more general (monotone) hard-sphere potentials. 
All the estimates will then depend on the specific balance between the parameter $\ep$ and the type of the singularity close to 1 encoded in
the function $p_\ep$.
\end{rmk}

%%%%%%%%%%%%%%%%%%%%%%%%%%%
\section{Proof of Theorem \ref{thm:main1}, existence of smooth approximate solutions}\label{Sec:3}
The first step of the proof is to construct local in time unique regular solution to \eqref{eq:AW-2}. Thanks to the presence of the approximation term $\vpe$ this can be done following the iterative scheme, described for example in Appendix B of Constantin et al. \cite{constantin2020}. The extension of this solution to the global in time solution requires some uniform (with respect to time) estimates that are presented below.

\subsection{Basic energy estimates}
In this section we assume that $\re,u_\ep$ are regular solutions to \eqref{eq:AW-2} in the time interval $[0,T]$ in the class \eqref{eq:regularity}, and that $\re$ is non-negative. For such solutions we first obtain straight from the continuity equation \eqref{eq:AW-2-cont} that
\eq{
\|\re\|_{L^1_x}(t)=\|\re^0\|_{L^1_x},
}
for all $t\in[0,T]$. Multiplying the momentum equation \eqref{eq:AW-2_mom} by $\ue$ and integrating by parts we obtain the classical energy estimate.
\begin{lem}\label{lem:energy}
For a regular solutions of system~\eqref{eq:AW-2}, we have
\begin{equation}\label{eq:energy}
\|\sqrt{\rho_\ep} u_\ep\|_{L^\infty_t L^2_x}^2 + \| \sqrt{\lambda_\ep(\rho_\ep)} \partial_x u_\ep \|_{L^2_t L^2_x}^2 \leq C E_{0,\ep}  ,
\end{equation}
with
\begin{equation*}
E_{0,\ep} := \|\sqrt{\rho_\ep^0} u_\ep^0\|_{L^2_x}^2.
\end{equation*}	
\end{lem}
Next energy estimate is an analogue of Bresch-Desjardins entropy for the compressible Navier-Stokes. We first introduce the quantity
\eq{H'_\ep(\rho_\ep) = p_\ep(\rho_\ep) + \varphi_\ep(\rho_\ep).}

\begin{lem}\label{lem:BD-entropy}
For a regular solution of system~\eqref{eq:AW-2}, we have
\begin{equation}\label{eq:BD-entropy}
\|\sqrt{\rho_\ep} w_\ep \|_{L^\infty_t L^2_x}^2 + \|H_\ep(\rho_\ep)\|_{L^\infty_t L^1_x} + \|\sqrt{\rho_\ep} \partial_x \big(p_\ep(\rho_\ep) + \varphi_\ep(\rho_\ep)\big) \|_{L^2_t L^2_x}^2 
\leq C E_{1,\ep} (1 + T),
\end{equation}
with 
\begin{equation*}
E_{1,\ep} := \|\sqrt{\rho_\ep^0} w_\ep^0 \|_{L^2_x}^2 + \|H_\ep(\rho_\ep^0)\|_{L^1_x}.
\end{equation*}	
\end{lem}

\begin{proof}
Recall that at the level of regular solutions, the system \eqref{eq:AW-2} can be reformulated  as \eqref{eq:AW-0} with
$w_\ep$ given by \eqref{w_ep}.
Therefore, multiplying the equation \eqref{eq:AW-0_mom} by $w_\ep$ and integrating by parts we easily  obtain 
\begin{equation*}
\|\sqrt{\rho_\ep} w_\ep \|_{L^\infty_t L^2_x} \leq \|\sqrt{\rho_\ep^0} w_\ep^0 \|_{L^2_x}.
\end{equation*}
Using again formula \eqref{w_ep} to substitute for $u_\ep$ in \eqref{eq:AW-2-cont} we obtain a porous medium equation for $\re$
\begin{equation}\label{eq:AW-3}
\partial_t \rho_\ep + \partial_x(\rho_\ep w_\ep) - \partial_x\Big(\rho_\ep \partial_x \big( p_\ep(\rho_\ep) + \varphi_\ep(\rho_\ep)\big) \Big) = 0. 
\end{equation}
Multiplying this equation by $H'_\ep(\rho) = p_\ep(\rho) + \varphi_\ep(\rho)$ and integrating over space and time, we get
\begin{align*}\
\sup_{t \in (0,T)} \int_{\T} H_\ep(\rho_\ep) dx + \|\sqrt{\rho_\ep} \partial_x \big(H'_\ep(\rho_\ep)\big) \|_{L^2_t L^2_x}^2 
& \leq \int_\T H_\ep(\rho_\ep^0)dx + C \|\sqrt{\rho_\ep} w_\ep\|_{L^2_t L^2_x}^2 \nonumber \\
& \leq \int_\T H_\ep(\rho_\ep^0)dx + C T \|\sqrt{\rho_\ep^0} w_\ep^0 \|_{L^2_x}^2.
\end{align*}
Note that $H_\ep$ consists of two components: positive but singular $H_\ep^s=\int_0^{\re}p_\ep(r)\,dr$ and the non-singular negative one $H_\ep^n=\int_0^\rho \varphi_\ep(r) dr $ since $\alpha < 1$. However can absorb this negative contribution in $H_\ep^{s}(\rho)$ as follows.
Let us first observe that
\[
H^s_\ep(r) \underset{r \to 1^-}\sim \dfrac{\ep}{(1-r)^{\beta-1}},
\]
and therefore, for any $C_1, C_2 > 0$ (independent of $\ep$) there exist $C_1', C_2'$ (also independent of $\ep$) such that
\begin{align*}
H^s_\ep(\rho_\ep) \mathbf{1}_{\{ \rho_\ep \geq 1 - C_1' \ep^{\frac{1}{\beta-1}} \} } \geq C_1 \rho_\ep^\alpha  \mathbf{1}_{\{ \rho_\ep \geq 1 - C_1' \ep^{\frac{1}{\beta-1}} \} }; \\
\int_\Omega H^s_\ep(\rho_\ep) \mathbf{1}_{\{ 1 - C_2' \ep^{\frac{1}{\beta-1}} \leq \rho_\ep \leq 1 - C_1' \ep^{\frac{1}{\beta-1}}  \} } dx \geq C_2 \ep.
\end{align*}
Hence, splitting the integral of $H^n_\ep$ into two parts, we get
\begin{align*}
\int_\Omega |H_\ep^n(\rho_\ep)|
& = \int_{\Omega} \dfrac{\ep}{\alpha (1-\alpha)} \rho_\ep^\alpha \mathbf{1}_{\{ \rho_\ep \leq 1 - C_1' \ep^{\frac{1}{\beta-1}} \} } dx
+  \int_{\Omega} \dfrac{\ep}{\alpha (1-\alpha)} \rho_\ep^\alpha \mathbf{1}_{\{ \rho_\ep \geq 1 - C_1' \ep^{\frac{1}{\beta-1}} \} } dx\\
& \leq \dfrac{\ep}{\alpha(1-\alpha)}|\Omega| 
+ \int_\Omega |H_\ep^s(\rho_\ep)| \mathbf{1}_{\{ \rho_\ep \geq 1 - C_1' \ep^{\frac{1}{\beta-1}} \} } dx \\
& \leq \int_\Omega |H_\ep^s(\rho_\ep)| \mathbf{1}_{\{ 1 - C_2' \ep^{\frac{1}{\beta-1}} \leq \rho_\ep \leq 1 - C_1' \ep^{\frac{1}{\beta-1}} \} } dx
+ \int_\Omega |H_\ep^s(\rho_\ep)| \mathbf{1}_{\{ \rho_\ep \geq 1 - C_1' \ep^{\frac{1}{\beta-1}} \} } dx
\end{align*}
for sufficiently large $C_1', C_2' >0$, independent of $\ep$.

This gives the result of the lemma.
\end{proof}

\subsection{Upper and lower bound on the density}
The purpose of this section is to prove that $\re$ is uniformly (in time) bounded from above by $ \urho_\ep$ and from below by $\lrho_\ep$. This will be done in two lemmas below.
\begin{lem}[Upper bound on the density]\label{lem:up_rho}
Let $T >0$, and let $(\rho_\ep, u_\ep)$ be a solution to \eqref{eq:AW-2} in the class \eqref{eq:regularity}, and satisfying the energy estimates \eqref{eq:energy} and \eqref{eq:BD-entropy}.
Assume moreover that initially $E_{0,\ep}$ and $E_{1,\ep}$, defined in Lemmas~\ref{lem:energy}-\ref{lem:BD-entropy}, are bounded uniformly with respect to $\ep$.
Then there exists a positive constant $C$ independent of $\ep$ and $T$ such that
\begin{equation}
\rho_\ep(t,x) \leq 1 - C\left(\dfrac{\ep}{1+\sqrt{T}}\right)^{\frac{1}{\beta-1}} =: \urho_\ep \qquad \forall \ t \in [0,T], \ x \in \mathbb{T} . 
\end{equation}
\end{lem}

\begin{proof}
First of all, the $L^\infty L^1$ bound~\eqref{eq:BD-entropy} on $H_\ep^s$ provides the upper bound
\[
\|\rho_\ep\|_{L^\infty_{t,x}} \leq 1.
\]
One can actually derive a more precise upper bound on the density thanks to the previous estimates.
First Nash' inequality reads:
\[
\|H_\ep^s(\rho_\ep(t,\cdot))\|_{L^2_x} 
\leq C \|H_\ep^s(\rho_\ep(t,\cdot))\|_{L^1_x}^{2/3} \|\partial_x H_\ep^s(\rho_\ep(t,\cdot))\|_{L^2_x}^{1/3} + C \|H_\ep^s(\rho_\ep(t,\cdot))\|_{L^1_x},
\]
and we have by Sobolev inequality
\begin{align}\label{eq:sob-Hep}
\|H_\ep^s(\rho_\ep(t,\cdot))\|_{L^\infty_x}
& \leq \|H_\ep^s(\rho_\ep(t,\cdot))\|_{L^2_x}^{1/2} \|\partial_x H_\ep^s(\rho_\ep(t,\cdot))\|_{L^2_x}^{1/2} + \|H_\ep^s(\rho_\ep(t,\cdot))\|_{L^2_x} \nonumber\\
& \leq C \Bigg(\|H_\ep^s(\rho_\ep(t,\cdot))\|_{L^1_x}^{1/3} \|\partial_x H_\ep^s(\rho_\ep(t,\cdot))\|_{L^2_x}^{2/3} 
				+ \|H_\ep^s(\rho_\ep(t,\cdot))\|_{L^1_x}^{1/2} \|\partial_x H_\ep^s(\rho_\ep(t,\cdot))\|_{L^2_x}^{1/2} \nonumber\\
& \qquad +  \|H_\ep^s(\rho_\ep(t,\cdot))\|_{L^1_x}^{2/3} \|\partial_x H_\ep^s(\rho_\ep(t,\cdot))\|_{L^2_x}^{1/3} +  \|H_\ep^s(\rho_\ep(t,\cdot))\|_{L^1_x} \Bigg) .
\end{align}
Now, we observe that the control of the energy~\eqref{eq:energy} together with the estimate~\eqref{eq:BD-entropy} provide the bound
\begin{equation}
\|\sqrt{\rho_\ep}p'_\ep(\rho_\ep) \partial_x \rho_\ep\|_{L^\infty_t L^2_x}^2
\leq C\big(E_{1,\ep}(1+T) + E_{0,\ep}\big),
\end{equation}
and on the other hand
\[
\|\sqrt{\rho_\ep}p'_\ep(\rho_\ep) \partial_x \rho_\ep\|_{L^\infty_t L^2_x}^2
= \left\|\sqrt{\rho_\ep}\dfrac{p_\ep'(\rho_\ep)}{p_\ep(\rho_\ep)}  \partial_x H^s_\ep(\rho_\ep) \right\|_{L^\infty_t L^2_x}^2.
\]
Since
\[
p'_\ep(\rho_\ep) 
= \ep \dfrac{\rho_\ep^{\gamma-1}}{(1-\rho_\ep)^{\beta+1}} \big[\gamma(1-\rho_\ep) + \beta \rho_\ep\big],
\]
we deduce that
\begin{align*}
\sqrt{\rho_\ep}\dfrac{p_\ep'(\rho_\ep)}{p_\ep(\rho_\ep)} 
& = \dfrac{\big[\gamma(1-\rho_\ep) + \beta \rho_\ep\big]\rho_\ep^{-1/2}}{1-\rho_\ep} \\
& = \left(\gamma + \beta \dfrac{\rho_\ep}{1-\rho_\ep} \right)\rho_\ep^{-1/2} \\
& \geq \gamma \|\rho_\ep\|_{L^\infty}^{-1/2}.
\end{align*}
Therefore
\begin{align*}
\|\partial_x H_\ep^s(\rho_\ep) \|_{L^\infty_t L^2_x} 
& \leq \gamma^{-1} \|\rho_\ep\|_{L^\infty_{t,x}}^{1/2} \|\sqrt{\rho_\ep}p'_\ep \partial_x \rho_\ep\|_{L^\infty_t L^2_x}\\ 
& \leq C \|\rho_\ep\|_{L^\infty_{t,x}}^{1/2} \sqrt{\big(E_{1,\ep}(1+T) + E_{0,\ep}\big)} \\
& \leq C \sqrt{\big(E_{1,\ep}(1+T) + E_{0,\ep}\big)},
\end{align*}
and so, coming back to~\eqref{eq:sob-Hep}:
\begin{equation}
\|H_\ep^s(\rho_\ep) \|_{L^\infty_{t,x}} \leq C \sqrt{\big(E_{1,\ep}(1+T) + E_{0,\ep}\big)}.
\end{equation}
Finally, it is easy to show that this latter bound yields
\[
\|\rho_\ep\|_{L^\infty_{t,x}} \leq 1 - C\left(\dfrac{\ep}{1+\sqrt{T}}\right)^{\frac{1}{\beta-1}},
\]
for some constant $C$ independent of $\ep$ and $T$, provided that $E_{1,\ep}$ and $E_{0,\ep}$ are bounded uniformly with respect to $\ep$.
\end{proof}

\begin{rmk}\label{rmk:deg-piep}
This upper bound ensures that $\rho_\ep$ never reaches $1$ at $\ep > 0$ fixed in finite time, and so the pressure remains bounded by a constant which depends a-priori on $\ep$. Indeed, we have
\begin{equation}\label{eq:deg-pep}
p_\ep(\rho_\ep(t,x))
\leq \dfrac{\ep}{(1 - \|\rho_\ep\|_{\infty})^\beta}
\leq C(1+\sqrt{T})^{\frac{\beta}{\beta-1}} \ep^{-\frac{1}{\beta-1}} \quad \forall \ t \in [0, T], \ x \in \mathbb{T}.
\end{equation}
Note that the same type of estimate holds for the potential $\pi_\ep(\rho_\ep)$.
\end{rmk}

\bigskip

In the next step we progress with the lower bound on the density. This is the only place where the artificial approximation term $\vpe$ matters.

\begin{lem}[Lower bound on the density]\label{lem:low_rho}
Let the assumptions of Lemma \ref{lem:up_rho} be satisfied. Then for a constant $C>0$ independent of $\ep$ and $T$ we have:
\begin{equation}\label{eq:low-bd-rho}
\left\|\dfrac{1}{\rho_\ep}\right\|_{L^\infty_{t,x}} 
\leq C \ep^{-\frac{2}{1-2\alpha}} \big(1+T\big)^{\frac{1}{1 - 2\alpha}} =: \dfrac{1}{\lrho_\ep}.
\end{equation}
\end{lem}

\begin{proof}
The estimate \eqref{eq:BD-entropy} ensures that $\sqrt{\rho_\ep} \partial_x \varphi_\ep(\rho_\ep)$ is controlled uniformly w.r.t. $\ep$ in $L^\infty L^2$, namely there exists $C > 0$ independent of $\ep$ and $T$ such that:
\[
\|\partial_x(\ep \rho_\ep^{\alpha-\frac{1}{2}}) \|_{L^\infty_t L^2_x}^2 
\leq C\big(E_{1,\ep}(1+T) + E_{0,\ep}\big).
\]
One the other hand, we have by conservation of mass
\[
\int_{\T} \rho_\ep(t,x) dx = \int_{\T} \rho_\ep^0(x) dx = M^0_\ep
\]
and by hypothesis we have
\[
0 < \underline{M}^0 \leq M^0_\ep < 1,
\]
for some $\underline{M}^0$ independent of $\ep$.
Therefore, for any time $t$, there exists some $\bar x(t) = \bar x (t,\ep) \in \Omega$ such that
\[
\rho_\ep (t,\bar x(t)) \geq \dfrac{\underline{M}^0}{|\T|}.
\]
For all $t\in(0,T]$ and $x \in \Omega$ we can write
\[
\ep \big( \rho_\ep(t,x) \big)^{\alpha-1/2} - \ep \big( \rho_\ep(t,\bar x(t)) \big)^{\alpha-1/2}
= \int_{\bar x(t)}^x \ep \partial_x \big( \rho_\ep(t,x) \big)^{\alpha-1/2},
\]
so that
\begin{align*}
\ep \big( \rho_\ep(t,x) \big)^{\alpha-1/2}
& \leq \ep\big( \rho_\ep(t,\bar x(t)) \big)^{\alpha-1/2}
+ |x-\bar x(t)|^{1/2} \|\partial_x(\ep \rho_\ep^{\alpha-\frac{1}{2}}) \|_{L^2_x} \\
& \leq \ep \left(\dfrac{\underline{M}^0}{|\T|}\right)^{\alpha-1/2} + |\T|^{1/2} \|\partial_x(\ep \rho_\ep^{\alpha-\frac{1}{2}}) \|_{L^2_x}.
\end{align*}
Hence, $\ep  \big( \rho_\ep(t,x) \big)^{\alpha-1/2} \leq C\sqrt{1+T}$ and finally
\[
\rho_\ep^{-1}(t,x) \leq C \ep^{-{\frac{1}{1/2 -\alpha}}}\big(1+T\big)^{\frac{1}{1 -2\alpha}} \leq C \ep^{-2} \big(1+T\big)^{\frac{1}{1 - 2\alpha}} \quad \forall \ t > 0, \ x \in \T.
\]
\end{proof}
\begin{rmk}
For sake of simplicity, we will sometimes estimate $\underline{\rho}_\ep$ as follows
\[
\underline\rho_\ep^{-1} = C \ep^{-{\frac{1}{1/2 -\alpha}}}\big(1+T\big)^{\frac{1}{1 -2\alpha}} \leq C_T \ep^{-2}.
\]
\end{rmk}

%%%%%%%%%%%%%%%%%%%%%%%%%%%
\subsection{Further regularity estimates}

\paragraph{Control of the singular diffusion}
In the next step we provide the estimates of 
\begin{equation}\label{df:Vep}
V_\ep := \lambda_\ep(\rho_\ep) \partial_x u_\ep,
\end{equation}
following the reasoning of Constantin et al.~\cite{constantin2020} ($V_\ep$ corresponds to what is called {\emph{active potential}} in~\cite{constantin2020}).

Let us first check that at the level of regular solutions what is the equation satisfied by $V_\ep$.
\begin{lem} \label{prop:Ve}
The variable $V_\ep$ satisfies
\begin{align}\label{eq:Vep}
\partial_t V_\ep + \left(u_\ep + \dfrac{\lambda_\ep(\rho_\ep)}{\rho_\ep^2}\partial_x \rho_\ep\right) \partial_x V_\ep - \dfrac{\lambda_\ep(\rho_\ep)}{\rho_\ep} \partial^2_x V_\ep 
= - \dfrac{\big( \lambda'_\ep(\rho_\ep) \rho_\ep + \lambda_\ep(\rho_\ep)\big)}{\big(\lambda_\ep(\rho_\ep)\big)^2} V_\ep^2.
\end{align}
\end{lem}
\begin{proof}
Dividing \eqref{eq:AW-2} by $\re>0$ and using the continuity equation we get
\eq{
\pt\ue+\ue\px\ue-\frac{1}{\re}\px V_\ep=0,}
and taking the space derivative
\eq{
\pt\px\ue+\px(\ue\px\ue)-\px\lr{\frac{1}{\re}\px V_\ep}=0.}
On the other hand multiplying \eqref{eq:AW-2-cont} by $\lambda_\ep'(\re)$ we get
\eq{
\pt\lambda_\ep(\re)+\px\lambda_\ep(\re)\ue+\lambda_\ep'(\re)\re\px\ue=0.}
Hence
\eq{
\pt V_\ep=&\lambda_\ep(\re)\pt\px\ue+\px\ue\pt\lambda_\ep(\re)\\
=&-\lambda_\ep(\re)\px\lr{\ue\px\ue}+\lambda_\ep(\re)\px\lr{\frac{1}{\re}\px V_\ep}-\px\lambda_\ep(\re)\ue\px\ue-\lambda_\ep'(\re)\re(\px\ue)^2\\
=&-\px(\ue V_\ep)+\px\lr{\frac{\lambda_\ep(\re)}{\re}\px V_\ep}-\lambda_\ep'(\re)\re\lr{\frac{V_\ep}{\lambda_\ep(\re)}}^2-\frac{\px \lambda_\ep(\re)}{\re}\px V_\ep\\
=&-\ue\px V_\ep-\frac{1}{\lambda_\ep(\re)} V_\ep^2+\px\lr{\frac{\lambda_\ep(\re)}{\re}\px V_\ep}-\lambda_\ep'(\re)\re\lr{\frac{V_\ep}{\lambda_\ep(\re)}}^2-\frac{\px \lambda_\ep(\re)}{\re}\px V_\ep,
}
from which \eqref{eq:Vep} follows.
\end{proof}

\begin{lem}\label{lem:V-LiL2} We have
\eq{\label{eq:est_V}
& \|V_\ep \|_{L^\infty_t L^2_x}^2 + \ep \|\partial_x V_\ep\|_{L^2_t L^2_x}^2 \\
& \leq  C\|V_\ep(0)\|_{L^2_x}^2 \exp \Bigg[
T \Big( \ep^{-3} \|R\|_{L^\infty_t L^2_x}^4 + \dfrac{\ep^{-1}}{\sqrt{\lrho_\ep}}\|\sqrt{\rho_\ep} u_\ep\|_{L^\infty_t L^2_x}^2\Big) + \dfrac{\ep^{\red{-3}}}{\lrho_\ep^{\frac{8}{3}\alpha} (1-\urho_\ep)^{\frac{4}{3}(\beta+2)}} \|V_\ep\|_{L^2_t L^2_x}^2 \Bigg] 
}
where 
\[
\|R\|_{L^\infty_t L^2_x}
:= {\left\| \dfrac{\partial_x \lambda_\ep(\rho_\ep)}{\rho_\ep} \right\|_{L^\infty_t L^2_x}}
\leq \dfrac{C}{\sqrt{\lrho_\ep}(1-\urho_\ep)} (1+ \sqrt{T}).
\]
\end{lem}

\begin{proof}
We multiply~\eqref{eq:Vep} by $V_\ep$ and integrate with respect to space to get
\begin{align*}
& \dfrac{d}{dt} \int \dfrac{V_\ep^2}{2} + \int \dfrac{\lambda_\ep(\rho_\ep)}{\rho_\ep} (\partial_x V_\ep)^2 \\
& = - \int \partial_x \left(\dfrac{\lambda_\ep(\rho_\ep)}{\rho_\ep}\right) V_\ep \partial_x V_\ep 
 - \int  \dfrac{\lambda_\ep(\rho_\ep)}{\rho_\ep^2}\partial_x \rho_\ep \partial_x V_\ep V_\ep 
 - \int u_\ep \partial_x V_\ep V_\ep \\
& \qquad - \int \dfrac{\big( \lambda'_\ep(\rho_\ep) \rho_\ep  + \lambda_\ep(\rho_\ep)\big)}{\big(\lambda_\ep(\rho_\ep)\big)^2} V_\ep^3 \\
& =- \int  \dfrac{\px\lambda_\ep(\rho_\ep)}{\rho_\ep}\partial_x \rho_\ep \partial_x V_\ep V_\ep 
 - \int u_\ep \partial_x V_\ep V_\ep - \int \dfrac{\big( \lambda'_\ep(\rho_\ep) \rho_\ep  + \lambda_\ep(\rho_\ep)\big)}{\big(\lambda_\ep(\rho_\ep)\big)^2} V_\ep^3\\
& = I_1 +I_2 + I_3.
\end{align*}
First, let us observe that
\begin{equation}
 \dfrac{\lambda_\ep(\rho_\ep)}{\rho_\ep} 
 = \rho_\ep p'_\ep(\rho_\ep) + \ep \rho^{\alpha-1} 
 \geq \ep
 \end{equation}
 \emph{Control of $I_1$}. We have
\begin{align*}
|I_1 |
& \leq \int  \underset{:= R}{\underbrace{\left|\dfrac{\partial_x \lambda_\ep(\rho_\ep)}{\rho_\ep}\right|  }} |\partial_x V_\ep| |V_\ep| \leq \|R\|_{L^2_x} \|\partial_x V_\ep\|_{L^2_x} \|V_\ep\|_{L^\infty_x}.
\end{align*}
Let us estimate $\|R\|_{L^\infty_t L^2_x}$. We denote 
\[
\lambda_\ep^s = \rho_\ep^2 p'_\ep(\rho_\ep), \qquad 
\lambda^n_\ep = \ep \rho_\ep^\alpha,
\]
We have one the one hand
\begin{align*}
\left\|\dfrac{\partial_x \lambda_\ep^s}{\rho_\ep} \right\|_{L^\infty_t L^2_x}
& = \left\| \dfrac{2 \rho_\ep p'_\ep (\rho_\ep) + \rho_\ep^2 p''_\ep(\rho_\ep)}{\rho_\ep} \partial_x \rho_\ep \right\|_{L^\infty_t L^2_x} \\
& \leq \left\|\dfrac{2 \sqrt{\rho_\ep} + \rho_\ep^{3/2} p''_\ep(\rho_\ep)/ p'_\ep(\rho_\ep)}{\rho_\ep}\right\|_{L^\infty} \|\sqrt{\rho_\ep} p'_\ep(\rho_\ep) \partial_x \rho_\ep \|_{L^\infty_t L^2_x}\\
& \leq C \left( \dfrac{1}{\sqrt{\lrho_\ep}} + \dfrac{1}{\sqrt{\lrho_\ep} (1- \urho_\ep)} \right)\|\sqrt{\rho_\ep} p'_\ep(\rho_\ep) \partial_x \rho_\ep \|_{L^\infty_t L^2_x},
\end{align*}
and
\begin{align*}
\left\|\dfrac{\partial_x \lambda_\ep^n}{\rho_\ep} \right\|_{L^\infty_t L^2_x}
& = \left\| \ep \alpha \rho_\ep^{\alpha-2} \partial_x \rho_\ep \right\|_{L^\infty_t L^2_x} \\
& \leq C\lrho_\ep^{-1/2}  \left\| \ep \rho_\ep^{\alpha - 3/2} \partial_x \rho_\ep \right\|_{L^\infty_t L^2_x}.
\end{align*}
% On the other hand
% \begin{align*}
% \left\| \dfrac{\lambda_\ep(\rho_\ep)}{\rho_\ep^2} \partial_x \rho_\ep \right\|_{L^\infty L^2}
% & \leq \dfrac{1}{\sqrt{\lrho_\ep}} \left(\|\sqrt{\rho_\ep} p'_\ep(\rho_\ep) \partial_x \rho_\ep \|_{L^\infty L^2}
% + \|\ep^{1-2\alpha} \rho_\ep^{\alpha-3/2} \partial_x \rho_\ep \|_{L^\infty L^2} \right),
% \end{align*}
so, using estimate \eqref{eq:BD-entropy}, we get
\begin{equation}\label{eq:estim-A}
\|R\|_{L^\infty_t L^2_x}
\leq C \dfrac{1 + \sqrt{T}}{\sqrt{\lrho_\ep}(1- \urho_\ep)}.
\end{equation}
Coming back to the control of $I_1$, we have:
\begin{align}\label{eq:I1-I_3}
|I_1|
& \leq \|V_\ep\|_{L^\infty_x} \|\partial_x V_\ep\|_{L^2_x} \|R\|_{L^2_x} \nonumber \\
& \leq C \|\partial_x V_\ep\|_{L^2_x} \Big(\|\partial_x V_\ep\|^{1/2} \|V_\ep\|_{L^2_x}^{1/2} + \|V_\ep\|_{L^2_x}\Big) \|R\|_{L^2_x} \nonumber\\
& \leq \dfrac{\bar{C} \ep}{4} \|\partial_x V_\ep\|_{L^2_x}^2 
    + C \Big( \ep^{-3}\|V_\ep\|_{L^2_x}^2 \|R\|_{L^2_x}^4 + \ep^{-1} \|V_\ep\|_{L^2_x}^2 \|R\|_{L^2_x}^2 \Big).
\end{align}
\emph{Control of $I_2$.}
\begin{align*}
|I_2| 
& \leq \int |u_\ep| |\partial_x V_\ep| |V_\ep| \\
& \leq \|\partial_x V_\ep \|_{L^2_x} \|V_\ep \|_{L^2_x} \|u_\ep \|_{L^\infty_x} \\
& \leq \dfrac{\bar{C}\ep}{4} \|\partial_x V_\ep\|_{L^2_x}^2 + C \ep^{-1} \|V_\ep\|_{L^2_x}^2 \|u_\ep\|_{H^1_x}^2,
\end{align*}
with
\eqh{
\|u_\ep\|_{H^1_x}^2 
\leq \|u_\ep\|_{L^2_x}^2 + \left\|\dfrac{1}{\lambda_\ep(\rho_\ep)}\right\|_{L^\infty}^2 \|V_\ep\|_{L^2_x}^2 
\leq \lrho_\ep^{-1/2} \|\sqrt{\rho_\ep} u_\ep\|_{L^2_x}^2 + \ep^{-2} \lrho_\ep^{-2\alpha} \|V_\ep\|_{L^2_x}^2\\
\leq \lrho_\ep^{-1/2} \|\sqrt{\rho_\ep} u_\ep\|_{L^2_x}^2 + \ep^{-2} \lrho_\ep^{-2\alpha} \|V_\ep\|_{L^2_x}^2
}
Hence
\begin{align}\label{eq:I_2}
|I_2| 
\leq \dfrac{\bar{C}\ep}{4} \|\partial_x V_\ep\|_{L^2_x}^2 
+ C \ep^{-1} \lrho_\ep^{-1/2} \|V_\ep\|_{L^2_x}^2 \|\sqrt{\rho_\ep} u_\ep\|_{L^2_x}^2
+ C \ep^{-3} \lrho_\ep^{-\alpha} \|V_\ep\|_{L^2_x}^4.
\end{align}
\emph{Control of $I_3$.}
\begin{align*}
|I_3| 
& \leq \int \left| \dfrac{\big( \lambda'_\ep(\rho_\ep) \rho_\ep  + \lambda_\ep(\rho_\ep)\big)}{\big(\lambda_\ep(\rho_\ep)\big)^2}\right| |V_\ep|^3 \\
& \leq \left\|\dfrac{\big( \lambda'_\ep(\rho_\ep) \rho_\ep  + \lambda_\ep(\rho_\ep)\big)}{\big(\lambda_\ep(\rho_\ep)\big)^2}\right\|_{L^\infty}
\big(\|\partial_x V_\ep\|_{L^2_x}^{1/2} \|V_\ep\|_{L^2_x}^{5/2} + \|V_\ep\|_{L^2_x}^3 \big)\\
& \leq C \dfrac{\ep^{-1}}{\lrho_\ep^{2\alpha} (1-\urho_\ep)^{\beta+2} }\big(\|\partial_x V_\ep\|_{L^2_x}^{1/2} \|V_\ep\|_{L^2_x}^{5/2} + \|V_\ep\|_{L^2_x}^3 \big) \\
& \leq \dfrac{\bar C \ep}{4} \|\partial_x V_\ep\|_{L^2_x}^2 
+ C \dfrac{\ep^{-4/3 -1/3}}{\big(\lrho_\ep^{2\alpha} (1-\urho_\ep)^{\beta+2}\big)^{4/3} } \|V_\ep\|_{L^2_x}^{5/2 \times 4/3} 
+ C \dfrac{\ep^{-1}}{\lrho_\ep^{2\alpha} (1- \urho_\ep)^{\beta+2}} \|V_\ep\|_{L^2_x}^3.
\end{align*}
Putting everything together, we get
\begin{align*}
& \dfrac{d}{dt} \|V_\ep\|_{L^2_x}^2 + C\ep \|\partial_x V_\ep\|_{L^2_x}^2 \\
& \leq C \Bigg(\ep^{{-3}} \|V_\ep\|_{L^2_x}^2 \|R\|_{L^2_x}^4 
+ \ep^{-1} \|V_\ep\|_{L^2_x}^2\|R\|_{L^2_x}^2 
+ \ep^{-1} \lrho_\ep^{-1/2} \|V_\ep\|_{L^2_x}^2 \|\sqrt{\rho_\ep} u_\ep\|_{L^2_x}^2 \\
& \qquad + \ep^{{-3}} \lrho_\ep^{-\alpha} \|V_\ep\|_{L^2_x}^4
+ \dfrac{\ep^{-5/3}}{\big(\lrho_\ep^{2\alpha} (1-\urho_\ep)^{\beta+2}\big)^{4/3} } \|V_\ep\|_{L^2_x}^{10/3} 
+  \dfrac{\ep^{-1}}{\lrho_\ep^{2\alpha} (1- \urho_\ep)^{\beta+2}} \|V_\ep\|_{L^2_x}^3
\Bigg) \\
& \leq C \|V_\ep\|_{L^2_x}^2 
\Bigg(\ep^{-3} \|R\|_{L^2_x}^4  
+ \ep^{-1} \lrho_\ep^{-1/2}  \|\sqrt{\rho_\ep} u_\ep\|_{L^2_x}^2 
+ \dfrac{\ep^{{-3}}}{\big(\lrho_\ep^{2\alpha} (1-\urho_\ep)^{\beta+2}\big)^{4/3} } \|V_\ep\|_{L^2_x}^2 \Bigg),
\end{align*}
and by Gronwall's inequality:
\begin{align*}
&\|V_\ep\|_{L^\infty_t L^2_x}^2 + \ep \|\partial_x V_\ep\|_{L^2_t L^2_x}^2 \\
&\leq C \|V_\ep(0)\|_{L^2_x}^2 
\exp\left[ T \left( \ep^{-3} \|R\|_{L^\infty_t L^2_x}^4 
+ \ep^{-1} \lrho_\ep^{-1/2}  \|\sqrt{\rho_\ep} u_\ep\|_{L^\infty_t L^2_x}^2 \right)
+ \dfrac{\ep^{{-3}}}{\big(\lrho_\ep^{2\alpha} (1-\urho_\ep)^{\beta+2}\big)^{4/3} } \|V_\ep\|_{L^2_t L^2_x}^2 
\right]
\end{align*}
\end{proof}

From this result, we infer estimates on $\partial_x u_\ep$. We have the following results:
\begin{lem}\label{lem:reg-pxu}
\begin{itemize} We have:\\

    \item[(i)]  $\partial_x u_\ep$ is bounded in $L^2_t L^2_x$ with the estimate
    \begin{align}
    \|\partial_x u_\ep \|_{L^2_t L^2_x}
& \leq \left\| \dfrac{1}{\sqrt{\lambda_\ep(\rho_\ep)}}\right\|_{L^\infty_{t,x}}  \|\sqrt{\lambda_\ep(\rho_\ep)}\partial_x u_\ep\|_{L^2 L^2}  
\leq C \ep^{-1/2} \lrho_\ep^{-\alpha/2} E_0^{1/2}.
    \end{align}
    \item[(ii)]  $\partial_x u_\ep$ is bounded in $L^\infty_t L^2_x$ with the estimate
\begin{align}\label{eq:reg-pxu}
\|\partial_x u_\ep \|_{L^\infty_t L^2_x}
& \leq \left\| \dfrac{1}{\lambda_\ep(\rho_\ep)}\right\|_{L^\infty_{t,x}} \|V_\ep\|_{L^\infty_t L^2_x}  \leq \ep^{-1} \lrho_\ep^{-\alpha} E_2^{1/2}.
\end{align}
    
    \item[(iii)] $\partial^2_x u_\ep$ is bounded in $L^2_t L^2_x$ with the estimate
\eq{\label{eq:reg-p2xu}
\|\partial^2_x u_\ep \|_{L^2_t L^2_x}
\leq &C \left\| \dfrac{1}{\lambda_\ep(\rho_\ep)}\right\|_{L^\infty_{t,x}}
\|\partial_x V_\ep\|_{L^2_t L^2_x} 
\\
& +  C \left\| \dfrac{1}{\lambda_\ep(\rho_\ep)}\right\|_{L^\infty_{t,x}}^2 \|\partial_x \lambda_\ep(\rho_\ep)\|_{L^\infty_t L^2_x} \big(\|\partial_x \lambda_\ep(\rho_\ep)\|_{L^\infty_t L^2_x} + 1\big) \|\partial_x u_\ep\|_{L^2_t L^2_x}  
}
\end{itemize}
\end{lem}

\begin{proof}
Part $(i)$ follows immediately  from \eqref{eq:energy}. Part $(ii)$ is combination of definition of $V_\ep$ and estimate \ref{eq:est_V}. And finally, part $(iii)$ follows by differentiating $V_\ep$, we get $\partial_x V_\ep = \partial_x \lambda_\ep \partial_x u_\ep + \lambda_\ep \partial^2_x u_\ep$, so that
\eq{
& \|\partial^2_x u_\ep\|_{L^2_t L^2_x} \\
& \leq \left\| \dfrac{1}{\lambda_\ep(\rho_\ep)}\right\|_{L^\infty_{t,x}} \Bigg( \|\partial_x V_\ep\|_{L^2_t L^2_x}
+ \|\partial_x \lambda_\ep\|_{L^\infty_t L^2_x} \|\partial_x u_\ep\|_{L^2_t L^\infty_x} \Bigg)\\
& \leq C \left\| \dfrac{1}{\lambda_\ep(\rho_\ep)}\right\|_{L^\infty_{t,x}} \Bigg( \|\partial_x V_\ep\|_{L^2_t L^2_x}
+ \|\partial_x \lambda_\ep\|_{L^\infty_t L^2_x}\lr{
\|\partial_x u_\ep\|_{L^2_t L^2_x}^{1/2}  \|\partial_x^2 u_\ep\|_{L^2_t L^2_x}^{1/2} +\|\partial_x u_\ep\|_{L^2_t L^2_x}}
\Bigg).
}
\end{proof}
\begin{comment}
    With this at hand we can get an analogue  of Lemma 4.2 from \cite{constantin2020}, namely
\begin{lem}
\begin{align}
& \|\partial^2_x \rho_\ep\|_{L^\infty L^2} + \|\partial_x V_\ep \|_{L^\infty L^2} + \|\partial^2_x V_\ep\|_{L^2 L^2} + \|\partial^2_x u_\ep \|_{L^\infty L^2} + \|\partial^3_x u_\ep\|_{L^2 L^2} \nonumber \\
& \leq C(\ep, \red{E_2}, \|\partial^2_x u_{0,\ep}\|_{L^2}, \|\partial^2_x \rho_{0,\ep}\|_{L^2}, \lrho_\ep, T). 
\end{align}
\end{lem}
\end{comment}

\paragraph{Higher order regularity estimates}
We begin with a formal computation for additional regularity estimates, for which we assume that the functions $\re,\ue$ are sufficiently smooth.
\begin{lem}\label{lem:formal}
Let $m\geq 2$ and $ \re \in C^1([0,T];C^m(\T))$, $ \ue \in C([0,T];C^{m+1}(\T))$ and they satisfy \eqref{eq:AW-2-cont}, then we have 
\begin{align}\label{rho-m-2}
	\begin{split}
	 \frac{1}{2}\frac{d}{dt} \Vert \partial_x^m \rho_\ep \Vert_{L^2_x}^2 & \leq C\left(    \Vert \px^m \ue \Vert_{L^2_x}  \Vert \px^m \re \Vert_{L^2_x}^2 + \Vert \re \Vert_{L^\infty_x} \Vert \px^m \re \Vert_{L^2_x}  \Vert \px^{m+1} \ue \Vert_{L^2_x}  \right)
 \end{split}
\end{align}
Moreover, for $l\geq 1$ and $ \re \in C([0,T];C^{l+1}(\T^d))$, $ \ue \in C([0,T];C^{l}(\T^d))$, $ V_\ep \in C^1([0,T];C^{l+2}(\T^d))$ satisfying \eqref{eq:Vep}, we have \begin{align}\label{V-m-1}
	\begin{split}
	 \frac{1}{2}\frac{d}{dt} \int \vert \partial_x^l V_\ep \vert^2 =& - \int \px^l \left( \left( \ue+  \dfrac{\lambda_\ep(\rho_\ep)}{\rho_\ep^2}\partial_x \rho_\ep\right) \partial_x V_\ep \right) \px^l V_\ep 
	  \\
	  &+\int \px^l \left(  \dfrac{\lambda_\ep(\rho_\ep)}{\rho_\ep} \partial^2_x V_\ep \right) \px^l V_\ep  - \int \px^l \left( \dfrac{\big( \lambda'_\ep(\rho_\ep) \rho_\ep + \lambda_\ep(\rho_\ep)\big)}{\big(\lambda_\ep(\rho_\ep)\big)^2} V_\ep^2 \right) \px^l V_\ep.
 \end{split}
\end{align}
\end{lem}
%\begin{rmk}
%The inequality \eqref{rho-m-2} remains true if we consider $ \re L^\infty(0,T;H^{m+1}(\T)) $ and $ \ue \in L^2(0,T;H^{m+1}(\T)) $.
%\end{rmk}
\begin{proof}

First, differentiating continuity equation $ m $ times with respect to $ x $ and multiplying by $\px^m \re  $, we deduce 
	\begin{align*}
		 \frac{1}{2}\frac{d}{dt} \int \vert \partial_x^m \rho_\ep \vert^2 = -\int \partial_x^m \rho_\ep \left( \partial_x^m(u_\ep \partial_x \rho_\ep )+ \partial_x^m(\rho_\ep \partial_x u_\ep) \right) .
	\end{align*}
We introduce the commutator notation $ [D,f]g = D(fg)- f Dg$, where $ D $ is any differential operator and $ f,g $ are sufficiently smooth functions, we rewrite the above equation as 
\begin{align}\label{rho-m-1}
	\begin{split}
	 \frac{1}{2}\frac{d}{dt} \int \vert \partial_x^m \rho_\ep \vert^2 =& -\int \partial_x^m \rho_\ep \left( [\px^m, \ue] \px \re  \right) - \int \px^m \re \ue \px^{m+1} \re \\
	 &-\int \partial_x^m \rho_\ep\left(  [\px^m, \re] \px \ue \right) - \int \px^m \re \re \px^{m+1} \ue .
 \end{split}
\end{align}

Now invoke the inequality for \cite[Page 12]{constantin2020}, it says that for $ g\in W^{n,2} (\T) $ with $ n\geq 3 $, we have 
\begin{align}\label{embed-1}
	\Vert \px g \Vert_{L^\infty_x} \leq C \Vert \px^2 g \Vert_{L^2_x} \leq C(n)  \Vert \px^n g \Vert_{L^2_x}.
\end{align}

Therefore, estimate \eqref{embed-1} and Kato-Ponce theory yield
\begin{align*}
	\Vert  [\px^m, \ue] \px \re  \Vert_{L^2_x} &\leq  C\left(\Vert \px^m \re \Vert_{L^2_x} \Vert \px \ue \Vert_{L^\infty_x} +  \Vert \px^m \ue \Vert_{L^2_x} \Vert \px \re \Vert_{L^\infty_x} \right) \leq C( \Vert \px^m \ue \Vert_{L^2_x} \Vert \px^m \re \Vert_{L^2_x})
\end{align*}
and 
\begin{align*}
	&\Vert  [\px^m, \re] \px \ue  \Vert_{L^2_x} \leq  C\left(\Vert \px^m \ue \Vert_{L^2_x} \Vert \px \re \Vert_{L^\infty_x} +  \Vert \px^m \re \Vert_{L^2_x} \Vert \px \ue \Vert_{L^\infty_x} \right)\leq C( \Vert \px^m \ue \Vert_{L^2_x} \Vert \px^m \re \Vert_{L^2_x}) .
\end{align*}
\begin{comment}
    \textbf{A justification of the commutator estimate:} At first we use chain rule to write 
\begin{align*}
	&[\px^m, \ue] \px \re = \sum_{i=1}^{m} C^1_i \px^i \re \;\px^{m+1-i} \ue, \\ 
	&[\px^m, \re] \px \ue = \sum_{i=1}^{m} C^2_i \px^i \ue \; \px^{m+1-i} \re,
\end{align*}
where $  C^1_i , C^2_i  $ are constants that depend only on $ m $. If $ i=0,m $, estimate the term with higher order derivative in $ L^2  $ and the other in $ L^\infty $. For $ 0<i<m $, estimate both terms in $L^\infty$. Then use \eqref{embed-1} to obtain the previous estimate.
\end{comment}

Now going back to \eqref{rho-m-1}, we notice that 
\begin{align*}
	- \int \px^m \re \ue \px^{m+1} \re = \frac{1}{2} \int \px \ue \vert \px^m \re \vert^2.
\end{align*}
This implies 
\begin{align*}
	\int \px \ue \vert \px^m \re \vert^2 \leq C  \Vert \px \ue \Vert_{L^\infty_x}  \Vert \px^m \re \Vert_{L^2_x}^2.
\end{align*}
Combining, all these estimates and plugging into \eqref{rho-m-1} we finally obtain \eqref{rho-m-2}.
\begin{comment}
    \begin{align}\label{rho-m-2}
	\begin{split}
	 \frac{1}{2}\frac{d}{dt} \Vert \partial_x^m \rho_\ep \Vert_{L^2_x}^2 & \leq C\left(    \Vert \px^m \ue \Vert_{L^2_x}  \Vert \px^m \re \Vert_{L^2_x}^2 + \Vert \re \Vert_{L^\infty_x} \Vert \px^m \re \Vert_{L^2_x}  \Vert \px^{m+1} \ue \Vert_{L^2_x}  \right)\\
	 &+ C\textcolor{red}{(\Vert \px^m \ue \Vert_{L^2_x} \Vert \px \re \Vert_{L^2_x}\Vert \px^m \re \Vert_{L^2_x} + \Vert \px \ue \Vert_{L^2_x} \Vert \px^m \re \Vert_{L^2_x}^2 ) }
 \end{split}
\end{align}
\end{comment}
Next, differentiating \eqref{eq:Vep} $ l $ times with respect to space and multiplying by $ \px^l  V_\ep  $ we obtain \eqref{V-m-1}.
\end{proof}

%Therefore, in order to close the estimate we need some additional information on $ \px^{m+1} \ue $.  

We already have the estimate of $ V_\ep $ in $ L^\infty_t L^2_x$ and $ \px V_\ep $ in $ L^2_t L^2_x$ from \eqref{eq:est_V}. Also, Lemma \ref{lem:up_rho} and Lemma \ref{lem:reg-pxu} give
\begin{align}
	\Vert \px \re \Vert_{L^\infty_t L^2_x} =\left(\alpha-\frac{1}{2}\right)^{-1} \frac{1}{\ep} \left\Vert \re ^{\frac{3}{2}-\alpha} \px(\ep \re^{\alpha- \frac{1}{2}}) \right\Vert_{L^\infty_t L^2_x} \leq \frac{1}{\ep}C(\bar{\rho}_\ep,\alpha, T, E_0, E_1). 
\end{align}
Using the formulas from Lemma \ref{lem:formal}, we want to derive the estimates for two further orders of regularity:
 \begin{itemize}
     \item First, the case corresponding to  $ m=2 $ and $ l=1 $ in \eqref{rho-m-2} and \eqref{V-m-1}, respectively.
     \item Then, the case  corresponding to $ m=3 $ and $ l=2 $ in \eqref{rho-m-2} and \eqref{V-m-1}, respectively.
 \end{itemize}
The obtained results are summarised in the two statements below.

\begin{prop}\label{prop-reg-m3}
Let $ \ep>0$ be fixed and $(\rho_\ep,u_\ep)$ be a regular solution of system~\eqref{eq:AW-2} with \[ E_{3,\ep} = E_{1,\ep} + \|\partial^2_x u_{0,\ep}\|_{L^2_x}+ \|\partial^2_x \rho_{0,\ep}\|_{L^2_x}+ \|\partial^3_x u_{0,\ep}\|_{L^2_x}+ \|\partial^3_x \rho_{0,\ep}\|_{L^2_x}. \] Then we have
\begin{align}
& \|\partial^3_x \rho_\ep\|_{L^\infty_t L^2_x} + \|\partial_x^{2} V_\ep \|_{L^\infty_t L^2_x} + \|\partial^3_x V_\ep\|_{L^2_t L^2_x} + \|\partial^3_x u_\ep \|_{L^\infty_t L^2_x} + \|\partial^{4}_x u_\ep\|_{L^2_t L^2_x} \nonumber \\
& \leq C(\ep, E_{3,\ep}, \lrho_\ep, \urho_\ep, T).
\end{align}

\end{prop}
The Lemma \ref{lem-reg-m2} is an analogue of \cite[ Lemma 4.2]{constantin2020} and the Proposition \ref{prop-reg-m3} is an analogue of \cite[Lemma 4.3]{constantin2020}. For completeness we include the proofs of both results, see the Appendix \ref{appendix}.
%%%%%%%%%%%%%%%%%%%%%%%%%%%%%
%%%%%%%%%%%%%%%%%%%%%%%%%%%%%
\section{Estimates uniform in $\ep$} \label{Sec:4}
In this section we first recall and derive additional uniform w.r.t. $\ep$ estimates which eventually will allow us to let $\ep\to 0$. 
The two key estimates of this section are: the control of the singular potential $\pi_\ep$, and the Oleinik entropy condition on $u_\ep$. 
\subsection{Estimates based on the energy bounds}\label{Sec:estim-ep}
First note that the Lemmas \ref{lem:energy} and \ref{lem:BD-entropy} are already uniform with respect to $\ep$, and so we have the following uniform bounds
\begin{prop}
Let the initial data $\rho_{\ep}^0,\ue^0$ be such that
\eq{E_{0,\ep}+E_{1,\ep} \leq C}
for some $C$ independent of $\ep$. Then we have for the solution $\re,\ue$:
\eq{\label{est_un_ep}
&\|\sqrt{\rho_\ep} u_\ep\|_{L^\infty_t L^2_x} \leq C\\
&\|\sqrt{\rho_\ep} w_\ep\|_{L^\infty_t L^2_x} \leq C\\
%&\|H_\ep^s(\re)\|_{L^\infty_t L^2_x} \leq C\\
%&\|H_\ep^n(\re)\|_{L^\infty_t L^2_x} \leq C\\
&\| \sqrt{\lambda_\ep(\rho_\ep)} \partial_x u_\ep \|_{L^2_t L^2_x}\leq C\\
%&\|\sqrt{\re}\px p_\ep(\re)\|_{L^2_t L^2_x}\leq C\\
%&\|\sqrt{\re}\px\varphi_\ep(\re)\|_{L^2_t L^2_x}\leq C\\
&\|\re\|_{L^\infty_t L^\infty_x}\leq 1.
}
\end{prop}

Note that so far we are lacking the uniform bound on the singular part of the potential $\pi_\ep(\re)$ (recall Remark~\ref{rmk:deg-piep} and estimate~\eqref{eq:deg-pep}). This is the purpose of the next lemma.
\begin{lem}\label{lem:uni_p}
Let the conditions of the previous proposition be satisfied and assume furthermore the condition~\eqref{hyp:mass}.
We have then
\begin{equation}\label{est_un_pi}
\|\pi_\ep(\rho_\ep)\|_{L^\infty_t L^1_x} + \|\partial_x \pi_\ep(\rho_\ep)\|_{L^\infty_t L^2_x} \leq C,
\end{equation} 
for a positive constant $C$ independent of $\ep$.
\end{lem}
\begin{proof}
We first multiply Eq~\eqref{eq:AW-2_mom} by $\psi(t,x) = \int_0^x\big(\rho_\ep(t,y) -<\rho_\ep>\big) dy$, where
\[
<\rho_\ep>  :=\frac{M^0_\ep}{|\T|},
\]
to get after time and space integration:
\eqh{
\int_0^t \int \rho_\ep^2 p'_\ep(\rho_\ep) \partial_x u_\ep \big(\rho_\ep - <\rho_\ep>\big) dx \, dt
= &-\int_0^t \int \rho_\ep^2 \varphi'_\ep(\rho_\ep) \partial_x u_\ep \big(\rho_\ep - <\rho_\ep>\big) dx \, dt\\
&+
\int_0^t \int(\rho_\ep u_\ep) \partial_t \psi \, dx \, dt
	+ \int_0^t \int \rho_\ep u_\ep^2 \big(\rho_\ep - <\rho_\ep>\big) dx \, dt.
}
From \eqref{est_un_ep} we infer that
\[
\left| \int_0^t \int \rho_\ep^2 p'_\ep(\rho_\ep) \partial_x u_\ep \big(\rho_\ep - <\rho_\ep>\big) dx \, dt \right| \leq C.
\]
Recall that we assumed that there exist $\overline{M}^0, \underline{M}^0$, independent of $\ep$ such that
\[
0 < \underline{M}^0 \leq M^0_\ep \leq \overline{M}^0.
\]
From the hypotheses~\eqref{hyp:rho}-\eqref{hyp:mass}, we ensure that
%we also assume that \red{add to the assumptions on the initial condition}
\begin{equation}\label{hyp:mean}
<\rho_\ep> \leq \hat \rho =\frac{\overline{M}^0}{|\T|}< 1,
\end{equation}
and we define $\rho^m_\ep = \dfrac{1 + <\rho_\ep>}{2} \leq \dfrac{1 + \hat\rho}{2}  $. 
We then split the previous integral into two parts, depending on the value of $\rho_\ep$. When $\rho_\ep \leq \rho^m_\ep$, the pressure remains far from the singularity uniformly with respect to $\ep$ and therefore
\begin{align*}
& \left| \int_0^t \int \rho_\ep^2 p'_\ep(\rho_\ep) \partial_x u_\ep \big(\rho_\ep - <\rho_\ep>\big) \mathbf{1}_{\{\rho_\ep \leq \rho_\ep^m\}} dx \, dt \right| \\
& \leq C \|\sqrt{p'_\ep(\rho_\ep)} \mathbf{1}_{\{\rho_\ep \leq \rho_\ep^m\}}\|_{L^2_t L^2_x} \|\rho_\ep \sqrt{p'_\ep(\rho_\ep)} \partial_x u_\ep \|_{L^2_t L^2_x} \\
& \leq C
\end{align*}
thanks to \eqref{est_un_ep}.
Hence
\begin{align*}
C & \geq \left| \int_0^t \int \rho_\ep^2 p'_\ep(\rho_\ep) \partial_x u_\ep \big(\rho_\ep - <\rho_\ep>\big) \mathbf{1}_{\{\rho_\ep > \rho_\ep^m\}} dx \, dt \right| \\
& \geq \frac{1-\hat \rho}{2}\left| \int_0^t \int \rho_\ep^2 p'_\ep(\rho_\ep) \partial_x u_\ep  \mathbf{1}_{\{\rho_\ep > \rho_\ep^m\}} dx \, dt \right|. 
\end{align*}
summarizing, we have shown that
\begin{equation}\label{p_control}
\left| \int_0^t\int_\T \rho_\ep^2 p'_\ep(\rho_\ep) \partial_x u_\ep  dx \, dt \right| \leq C.
\end{equation}
Next, note that for $\ep$ fixed, $s_\ep=\re p_\ep(\re)$ satisfies the following equation
\begin{equation}\label{eq:potential}
\partial_t s_\ep(\rho_\ep) + \partial_x (s_\ep(\rho_\ep) u_\ep) = - \rho_\ep^2 p'_\ep(\rho_\ep) \partial_x u_\ep.
\end{equation}
From this and the estimate \eqref{p_control}, we can deduce that 
\begin{equation*}
\|s_\ep(\rho_\ep) \|_{L^\infty_t L^1_x} \leq C.
\end{equation*}
Now, considering the cases of $\re$ far away from 1 and close to 1 separately and observing that $s_\ep(\rho)$ and $\pi_\ep(\rho)$ have the same singularity when $\rho$ is close to $1$, we obtain
\begin{equation*}
 \|\pi_\ep(\rho_\ep) \|_{L^\infty_t L^1_x} \leq C.
\end{equation*}
Regarding the control of the gradient $\partial_x \pi_\ep(\rho_\ep)$, we simply use estimate~\eqref{est_un_ep} to write
\begin{align*}
\|\partial_x \pi_\ep(\rho_\ep)\|_{L^\infty_t L^2_x}
& \leq \|\sqrt{\rho_\ep} \|_{L^\infty_{t,x}} 
\|\sqrt{\rho_\ep} \partial_x \big(p_\ep(\rho_\ep) + \varphi_\ep(\rho_\ep) \big) \|_{L^\infty_t L^2_x} \\
& \leq \|\sqrt{\rho_\ep} \|_{L^\infty_{t,x}}
\Big(\|\sqrt{\rho_\ep} w_\ep \|_{L^\infty_t L^2_x} + \|\sqrt{\rho_\ep} u_\ep \|_{L^\infty_t L^2_x}\Big)
\leq C.
\end{align*}
\end{proof}

\bigskip
From the control~\eqref{p_control}, we directly infer the next result.
\begin{cor}\label{cor:control-unif-V}
Under the same hypotheses, we have
\begin{equation}
\|\lambda_\ep(\rho_\ep) \partial_x u_\ep\|_{L^1_{t,x}} \leq C,
\end{equation}
for some $C> 0$ independent of $\ep$.
\end{cor}

\subsection{One-sided Lipschitz condition on $\partial_x u_\ep$}
The purpose of this subsection is to prove that $u_\ep$ satisfies the {\emph{Oleinik entropy condition}}, i.e. $\partial_x u_\ep \leq 1/t$. 
\begin{lem}\label{lem:Oleinik}
Let $\re,\ue$ be solution to system \eqref{eq:AW-2} and assume that initially
\eq{\esssup {(\lambda_\ep(\rho_\ep^0)\partial_x u_\ep^0)}\leq C}
uniformly w.r.t. $\ep$, and let us denote
\eq{ A = \max (\esssup {(\lambda_\ep(\rho_\ep^0)\partial_x u_\ep^0)}, 0).}
Then we have the $\ep-$uniform estimate:
\eq{
\px\ue (t,x) \leq \frac{A}{At+1}\leq \frac1t \qquad \forall \ t \in \ ]0,T], ~ x \in \mathbb{T}.
}

%uniformly w.r.t. $\ep$.
\end{lem}
\begin{proof}
The starting point is derivation of equation for 
\[
\tV_\ep = (At +1)\frac{V_\ep}{\underline{\lambda}_\ep}, 
\quad \mbox{where} \ \underline{\lambda}_\ep=\ep\lrho_\ep^\alpha.
\]
Similarly to proof of Proposition \ref{prop:Ve} we can show that
\eq{\label{tVe_conserv}
&\pt \tV_\ep+\px(\ue \tV_\ep)-\px\lr{\frac{\lambda_\ep(\re)}{\re}\px \tV_\ep}\\
&=\frac{A}{At+1}\tV_\ep-\frac{\lambda_\ep'(\re)\re \underline{\lambda}_\ep}{At+1}\lr{\frac{\tV_\ep}{\lambda_\ep(\re)}}^2-\frac{\px \lambda_\ep(\re)}{\re}\px \tV_\ep,
}
which for $\ep$ fixed holds pointwisely. We now derive the renormalised equation for $\tV_\ep$. To this purpose we multiply \eqref{tVe_conserv} by $S'(\tV_\ep)$, where $S$ is smooth, increasing and convex function, we obtain
\eq{\label{S_conserv}
&\pt S(\tV_\ep)+\px(\ue S(\tV_\ep))-\px\lr{\frac{\lambda_\ep(\re)}{\re}S'(\tV_\ep)\px \tV_\ep}\\
&=\lr{S(\tV_\ep)-S'(\tV_\ep)\tV_\ep}\px\ue
-S''(\tV_\ep)\frac{\lambda_\ep(\re)}{\re}\lr{\px \tV_\ep}^2\\
&+\frac{A}{At+1}S'(\tV_\ep)\tV_\ep
-\frac{S'(\tV_\ep)\lambda_\ep'(\re)\re \underline{\lambda}_\ep}{At+1}\lr{\frac{\tV_\ep}{\lambda_\ep(\re)}}^2-\frac{\px \lambda_\ep(\re)}{\re}\px S(\tV_\ep)
}
We set $S(y) = F_\eta(y)$ where $F_\eta$, $\eta > 0$, is a regularization of $(\cdot - A)_+$:
\begin{equation}\label{def:Fe}
F_\eta(y) =
\begin{cases}
0 \quad & \text{if}~ y \leq \eta \\
\dfrac{y-\eta}{2} + \dfrac{\eta}{2\pi} \sin(\pi \frac{y}{\eta}) \quad & \text{if}~ \eta \leq y \leq 2\eta \\
y- \dfrac{3\eta}{2} \quad & \text{if}~  y \geq 2\eta
\end{cases}
\end{equation}
For $\eta >0$ fixed, $F''_\eta \geq 0$, $0 \leq F'_\eta \leq 1$, and
\eq{\label{prop_kappa}
|F_\eta(y) - (y-A) F'_\eta(y)| \leq \lr{\frac{3}{2} + \frac{1}{2\pi}}\eta = \kappa \eta.
}
Note that for such choice of $S$ the second and the fourth terms on the r.h.s. of \eqref{S_conserv} are non-positive. Therefore, integrating \eqref{S_conserv} in space, we then get:
\eq{\label{ineq:renorm}
 \frac{d}{dt}\int F_\eta(\tV_\ep(t))  &\leq \int\Big|F_\eta(\tV_\ep) - F_\eta'(\tV_\ep)(\tV_\ep - A) \Big|  |\px \ue|\\
&\quad+\int \frac{A}{At+1}F'_\eta(\tV_\ep)\tV_\ep \left|1-\frac{1}{\lambda_\ep(\re)}\right|+ \int\left|\px\lr{\frac{\px\lambda_\ep(\re)}{\re}}\right|F_\eta(\tV_\ep).
}
The first term on the r.h.s. can be controlled using \eqref{prop_kappa}, \eqref{est_un_ep}, and the $\ep$-dependent bound from below for $\re$ \eqref{eq:low-bd-rho}, we obtain
\eq{
\int\Big|F_\eta(\tV_\ep) - F_\eta'(\tV_\ep)(\tV_\ep - A) \Big|  |\px \ue|\leq \kappa\eta\int |\px\ue|\leq \kappa C(T)\eta\ep^{-\alpha'}
}
with some $\alpha'>0$.
For the second term on the r.h.s. of \eqref{ineq:renorm} we compute
\begin{equation}
y F_\eta'(y) =
\begin{cases}
0 \quad & \text{if}~ y \leq \eta \\
\dfrac{y}{2} + \dfrac{y}{2} \cos(\pi \frac{y}{\eta}) \quad & \text{if}~ \eta \leq y \leq 2\eta \\
y \quad & \text{if}~  y \geq 2\eta,
\end{cases}
\end{equation}
and so $yF_\eta'(y)\leq 4 F_\eta(y)$ which can be then controlled by the Gronwall argument, since 
\eq{
\left\|1-\frac{1}{\lambda_\ep(\re)}\right\|_{L^\infty_t L^\infty_x}\leq C(T)\ep^{-\alpha''},
}
for some $\alpha''>0$.\\
Finally, the last term on the r.h.s. of \eqref{ineq:renorm} can also be controlled by the Gronwall's argument seeing that from Proposition~\ref{prop-reg-m3}:
\eq{
\left\|\px\lr{\frac{\px\lambda_\ep(\re)}{\re}}\right\|_{{L^1_t L^\infty_x}}\leq C(T)\ep^{-\alpha'''}}
for some $\alpha'''>0$.
Putting this together, and applying the Gronwall's inequality to \eqref{ineq:renorm} we obtain
\eq{\label{est:Fn}
&\int F_\eta(\tV_\ep(t))\\
&\leq \exp{\lr{\left\|\px\lr{\frac{\px\lambda_\ep(\re)}{\re}}\right\|_{L^1_t L^\infty_x}\!\!\!\!
+\log(At+1)\left\|1-\frac{1}{\lambda_\ep(\re)}\right\|_{L^\infty_t L^\infty_x}}}\int \left[F_\eta(\tV_\ep(0))+C(T)\kappa \eta t \ep^{-\alpha'}\right]\\
&\leq \exp{\lr{C(T)\ep^{-\alpha'''}
+\log(At+1)C(T)\ep^{-\alpha''}}}\left[\int F_\eta(\tV_\ep(0))+C(T)\kappa \eta t \ep^{-\alpha'}\right]
}
Passing to the limit $\eta\to 0$, we obtain
\eq{
&\int(\tV_\ep(t)-A)_+\leq \exp{\lr{C(T)\ep^{-\alpha'''}
+\log(At+1)C(T)\ep^{-\alpha''}}} \int(\tV_\ep(0)-A)_+.
}
Noticing that $\tV_\ep(0)=V_\ep(0)$ we therefore obtain
\eq{\int_\Omega(\tV_\ep(t)-A)_+\leq 0}
uniformly w.r.t. $\ep$ which implies
\eq{\label{almost_Oleinik}
\frac{\lambda_\ep(\re)}{\underline{\lambda}_\ep}\px\ue\leq \frac{A}{At+1}\leq \frac1t.}
For each $\ep>0$ we have ${\lambda_\ep(\re)}\geq{\underline{\lambda}_\ep}$, and  $\frac{\underline{\lambda}_\ep}{\lambda_\ep(\re)}\leq 1$, thus we get the required estimate uniformly in $\ep$.
\end{proof}
As a consequence of Lemma \ref{lem:Oleinik} and due to periodicity of the domain, we can  control the whole norm of the velocity gradient. 
\begin{cor}
We have
\eq{\label{est_uni_grad}
\|\partial_x u_\ep\|_{L^\infty_tL^1_x}\leq C}
for constant $C$ independent of $\ep$.
\end{cor}
\begin{proof}
Let us denote $D_\ep : = (\partial_x u_\ep)_+$.
We have, for any $t$
\begin{align*}
    %TV(u_\ep(t,\cdot)) 
    & \int_\T |\partial_x u_\ep(t,x)| dx \\
    %& \leq \int_\T \Big( |\partial_x u_\ep(t,x) - D_\ep(t,x)| + D_\ep(t,x) \Big) dx \\
    &  = \int_\T \Big( 2D_\ep(t,x) - \partial_x u_\ep(t,x)  \Big) dx \\
    &= \int_\T 2D_\ep(t,x) \,  dx\\
    & \overset{\text {Oleinik cond.}}{\leq} 2 |\T|\dfrac{A}{At+1} . 
\end{align*}
Taking the supremum  w.r.t. $t$ we conclude the proof.
\end{proof}

\section{Passage to the limit $\ep\to0$}\label{Sec:5}
The purpose of this section is to prove Theorem \ref{thm:main2}. We will show that when $\ep\to 0$ the sequence of solutions $\re,\ue$ gives rise to a sequence $\re,m_\ep,\pi_\ep$ converging to $\rho,m,\pi$ distributional solution of \eqref{eq:limit}.

\begin{proof}[Proof of Theorem \ref{thm:main2}]
Thanks to the uniform bounds from the previous section, there exist $\rho \in [0,1]$, $u$, and $\pi \geq 0$ such that
\begin{align*}
&\rho_\ep \rightharpoonup \rho \quad \text{weakly-* in} \quad L^\infty((0,T) \times \T), \\
&\ue\rightharpoonup u\quad \text{weakly-* in} \quad L^\infty ((0,T) \times \T), \\
%&\rho_\ep u_\ep \rightarrow m \quad \text{weakly-* in} \quad L^\infty ((0,T) \times \T), \\
&\pi_\ep(\rho_\ep) \rightarrow \pi \quad \text{weakly-* in} \quad L^\infty (0,T; H^1(\T)),
\end{align*}
up to selection of a subsequence.\\
We can immediately justify that
\eq{\label{strong_cond}
(1-\rho_\ep)\pi_\ep(\rho_\ep) \rightarrow 0 \quad \text{strongly in } ~ L^q((0,T)\times \T), ~q > 1, 
}
% This also implies \[ E_\ep(\rho_\ep)  \rightarrow 0 \quad \text{strongly in some} ~ L^q((0,T)\times(0,1)), ~q > 1.  \]
% Next, from $ 0\leq \rho \leq 1 $ we obtain 
and that the approximate viscosity term converges to $0$ strongly, i.e. 
\[  \rho_\ep \varphi_\ep(\rho_\ep)  \rightarrow 0 \quad \text{strongly in } L^\infty((0,T)\times \T). \]

To pass to the limit in the nonlinear terms we first use the continuity and momentum equations of system \eqref{eq:AW-2} to deduce that for any $p<\infty$ we have
\eq{\label{est_der_time}
\|\pt\rho_\ep\|_{L^\infty_t W^{-1,p}_x}+\|\pt(\rho_\ep u_\ep)\|_{L^\infty_t W^{-1,2}_x}\leq C,
}
where to estimate the time derivative of momentum, we use that $\lambda_\ep(\rho_\ep)\approx \pi_\ep(\rho_\ep)^{1+\frac1\beta}$, along with uniform estimates \eqref{est_un_ep} and \eqref{est_un_pi}.\\
Combining the control of $\partial_t \rho_\ep$ with the control of $\partial_x \pi_\ep(\rho_\ep)$ we can apply the standard compensated compactness argument (see,  Lemma 5.1 from~\cite{lions1996}) to justify that
\[
(1-\rho_\ep)\pi_\ep(\rho_\ep) \rightharpoonup (1-\rho)\pi \quad \text{in}~ \mathcal{D}',
\]
and so, from \eqref{strong_cond}, we deduce that $(1-\rho)\pi = 0$ a.e. in $(0,T)\times\mathbb{T}$ with $1- \rho \geq 0$, $\pi \geq 0$.\\

Similarly, combining the control of gradient of velocity \eqref{est_uni_grad} with the uniform estimates for the time derivatives \eqref{est_der_time} we can  justify that
\eq{
\rho_\ep u_\ep\to \rho u\quad \mbox{and}\qquad
\rho_\ep u_\ep^2\to \rho u^2
}
in the sense of distributions.

Finally, we can use the equation for $\pt\pi_\ep(\rho_\ep)$ (which is of the form \eqref{eq:potential}) to deduce that
\eq{
\pt\px \pi_\ep(\rho_\ep)\in L^\infty_t W^{-1,1}_x,
}
so, repeating the previous argument we can justify that also
\eq{
 u_\ep \pi_\ep(\rho_\ep)\to  u \pi
}
in the sense of the distributions.

\medskip The last part is to verify the entropy conditions for the limiting system.
First, it is clear that the one-sided Lipschitz estimate holds on the limit velocity $u$:
\[
\partial_x u \leq \frac{1}{t} \qquad \text{in} ~ \mathcal{D}'.
\]
Next, we write that for fixed $\ep$, smooth function $S$:
\begin{align*}
\partial_t(\rho_\ep S(u_\ep)) + \partial_x(\rho_\ep u_\ep S(u_\ep)) -\partial_x \big(S'(u_\ep)\lambda_\ep(\rho_\ep) \partial_x u_\ep \big)\\
= S''(u_\ep) \lambda_\ep(\rho_\ep) (\partial_x u_\ep)^2,
\end{align*}
hence, for convex function $S$:
\[
\partial_t(\rho_\ep S(u_\ep)) + \partial_x(\rho_\ep u_\ep S(u_\ep)) -\partial_x \big(S'(u_\ep)\lambda_\ep(\rho_\ep) \partial_x u_\ep \big)
\leq 0.
\]
As previously, we pass to the limit in the sense of distribution in the first two nonlinear terms thanks to compensated compactness arguments.
Next, since $(\lambda_\ep(\rho_\ep) \partial_x u_\ep)_\ep$ is bounded in $L^1_{t,x}$ it converges to some $\Lambda \in \mathcal{M}((0,T) \times \T)$.
Recall that $(u_\ep)_\ep$ is bounded in $L^\infty_{t,x}$, so $(S'(u_\ep)\lambda_\ep(\rho_\ep) \partial_x u_\ep)_\ep$ is bounded in $L^1_{t,x}$ and converges to some $\Lambda_S \in \mathcal{M}((0,T) \times \T)$, where $|\Lambda_S| \leq \mathrm{Lip}_S |\Lambda|$. Finally, we have proven that:
\[
\partial_t(\rho S(u)) + \partial_x(\rho u S(u)) -\partial_x \Lambda_S
\leq 0.
\]

\bigskip

The proof of the Theorem \ref{thm:main2} is therefore complete. 
\end{proof}

\bigskip

\newpage

\section{Appendix}\label{appendix}
In order to proof the Proposition \ref{prop-reg-m3}, we first state  the following lemma.
\begin{lem}\label{lem-reg-m2}
Let $ \ep>0$ be fixed and $(\rho_\ep,u_\ep)$ be a regular solution of system~\eqref{eq:AW-2} with
\[ E_{2,\ep} = E_{1,\ep} + \|\partial^2_x u_{0,\ep}\|_{L^2_x}+ \|\partial^2_x \rho_{0,\ep}\|_{L^2_x}. \] Then we have

\begin{align}
& \|\partial^2_x \rho_\ep\|_{L^\infty_t L^2_x} + \|\partial_x V_\ep \|_{L^\infty_t L^2_x} + \|\partial^2_x V_\ep\|_{L^2_t L^2_x} + \|\partial^2_x u_\ep \|_{L^\infty_t L^2_x} + \|\partial^3_x u_\ep\|_{L^2_t L^2_x} \nonumber \\
& \leq C(\ep, E_{2,\ep}, \lrho_\ep, \urho_\ep, T).
\end{align}

\end{lem}
\begin{proof}
Here, we recall that our estimate of $ \px \re $ in $ L^\infty L^2 $ :
\begin{align*}
	\Vert \px \re \Vert_{L^\infty_t L^2_x} =\left(\alpha-\frac{1}{2}\right)^{-1} \frac{1}{\ep} \left\Vert \re ^{\frac{3}{2}-\alpha} \px(\ep \re^{\alpha- \frac{1}{2}}) \right\Vert_{L^\infty_t L^2_x} \leq \frac{1}{\ep}C(\bar{\rho}_\ep,\alpha, T, E_0, E_1). 
\end{align*}
We consider the case that corresponds to $ m=2 $ in \eqref{rho-m-2}. Therefore, we have 
\begin{align*}
	\frac{1}{2}\frac{d}{dt} \Vert \partial_x^2 \rho_\ep \Vert_{L^2_x}^2 & \leq C\left(    \Vert \px^2 \ue \Vert_{L^2_x}  \Vert \px^2 \re \Vert_{L^2_x}^2 +  \Vert \px^2 \re \Vert_{L^2_x}  \Vert \px^{3} \ue \Vert_{L^2_x}  \right).
\end{align*}
On the other hand, using integration by parts for $ l=1 $ in \eqref{V-m-1} gives us
\begin{align*}
	 &\frac{1}{2}\frac{d}{dt} \int \vert \partial_x V_\ep \vert^2 + \int   \dfrac{\lambda_\ep(\rho_\ep)}{\rho_\ep} \vert  \px^2 V_\ep \vert^2\\
	 & = \int   \left( \ue+  \dfrac{\lambda_\ep(\rho_\ep)}{\rho_\ep^2}\partial_x \rho_\ep\right) \partial_x V_\ep  \px^2 V_\ep  + \int   \dfrac{\big( \lambda'_\ep(\rho_\ep) \rho_\ep + \lambda_\ep(\rho_\ep)\big)}{\big(\lambda_\ep(\rho_\ep)\big)^2} V_\ep^2  \px^2 V_\ep.
\end{align*}
From the observation \[ \dfrac{\lambda_\ep(\rho_\ep)}{\rho_\ep} 
= \rho_\ep p'_\ep(\rho_\ep) + \ep \rho^{\alpha-1} 
\geq \ep \]
we deduce that
\begin{align*}
	&\frac{1}{2}\frac{d}{dt} \int \vert \partial_x V_\ep \vert^2 + \ep \int  \vert  \px^2 V_\ep \vert^2\\
	&\leq  \int    \ue \partial_x V_\ep  \px^2 V_\ep  +\int   \dfrac{\lambda_\ep(\rho_\ep)}{\rho_\ep^2}\partial_x \rho_\ep \partial_x V_\ep  \px^2 V_\ep  + \int   \dfrac{\big( \lambda'_\ep(\rho_\ep) \rho_\ep + \lambda_\ep(\rho_\ep)\big)}{\big(\lambda_\ep(\rho_\ep)\big)^2} V_\ep^2  \px^2 V_\ep
	\;\; := \sum_{i=1}^{3}J_i
\end{align*}
\par
\emph{Control for $ J_1 $:} 
We proceed similarly as in the case of $ I_2 $ of Lemma~\ref{lem:V-LiL2}. We have 
\begin{align*}
	\vert J_1 \vert \leq \Vert \ue \Vert_{L^\infty_x} \Vert \px V_\ep \Vert_{L^2_x} \Vert \px^2 V_\ep \Vert_{L^2_x} \leq \dfrac{\ep}{16}  \Vert \px^2 V_\ep \Vert_{L^2_x}^2 + \frac{4}{\epsilon} \Vert \ue \Vert_{L^\infty_x}^2 \Vert \px V_\ep \Vert_{L^2_x}^2.
\end{align*}
We recall
\begin{align*}
	\|u_\ep\|_{H^1_x}^2	\leq \lrho_\ep^{-1/2} \|\sqrt{\rho_\ep} u_\ep\|_{L^2_x}^2 + \ep^{-2} \lrho_\ep^{-2\alpha} \|V_\ep\|_{L^2_x}^2
\end{align*}
to conclude 
\begin{align}\label{J1}
	\vert J_1 \vert \leq \dfrac{\ep}{16}  \Vert \px^2 V_\ep \Vert_{L^2_x}^2 + 4\left(\lrho_\ep^{-1/2} \ep^{-1} \|\sqrt{\rho_\ep} u_\ep\|_{L^2_x}^2 + \ep^{-3} \lrho_\ep^{-2\alpha} \|V_\ep\|_{L^2_x}^2 \right)  \Vert \px V_\ep \Vert_{L^2_x}^2.
\end{align}
\par
\emph{Control for $ J_2 $:} For this term, we observe
\begin{align*}
	\vert J_2 \vert \leq C(\bar{\rho}_\ep, \lrho_\ep)\Vert \px \re \Vert_{L^\infty_x} \Vert \px V_\ep \Vert_{L^2_x} \Vert \px^2 V_\ep \Vert_{L^2_x} \leq \dfrac{\ep}{16}  \Vert \px^2 V_\ep \Vert_{L^2_x}^2 + C(\ep,\bar{\rho}_\ep, \lrho_\ep) \Vert \px \re \Vert_{L^\infty_x}^2 \Vert \px V_\ep \Vert_{L^2_x}^2.
\end{align*}
Moreover, we have 
\begin{align}\label{J2}
	\vert J_2 \vert \leq \dfrac{\ep}{16}  \Vert \px^2 V_\ep \Vert_{L^2_x}^2 + C(\ep,\bar{\rho}_\ep, \lrho_\ep) \Vert \px^2 \re \Vert_{L^2_x}^2 \Vert \px V_\ep \Vert_{L^2_x}^2.
\end{align}
\par
\emph{Control for $ J_3 $:} The bound of $ \left\|\dfrac{\big( \lambda'_\ep(\rho_\ep) \rho_\ep  + \lambda_\ep(\rho_\ep)\big)}{\big(\lambda_\ep(\rho_\ep)\big)^2}\right\|_{L^\infty_x} $ implies 
\begin{align*}
	\vert J_3 \vert \leq \dfrac{\ep}{16}  \Vert \px^2 V_\ep \Vert_{L^2_x}^2 + C(\ep,\bar{\rho}_\ep, \lrho_\ep) \int \vert V_\ep\vert ^4 
\end{align*}
Moreover, using Nash inequality we obtain
\begin{align*}
	\Vert V_\ep \Vert_{L^4_x}^4 \leq C\left(	\Vert V_\ep \Vert_{L^2_x}^3 	\Vert \px V_\ep \Vert_{L^2_x} + 	\Vert V_\ep \Vert_{L^2_x}^4 \right) .
\end{align*}
Thus we deduce
 \begin{align}\label{J3}
	\vert J_3 \vert \leq \dfrac{\ep}{16}  \Vert \px^2 V_\ep \Vert_{L^2_x}^2 + C(\ep,\bar{\rho}_\ep, \lrho_\ep) \left(	\Vert V_\ep \Vert_{L^2_x}^3 	\Vert \px V_\ep \Vert_{L^2_x} + 	\Vert V_\ep \Vert_{L^2_x}^4 \right).
\end{align}
Therefore, combining \eqref{J1}, \eqref{J2} and \eqref{J3}, it yields
\begin{align}\label{V2}
	\begin{split}
	&\frac{1}{2}\frac{d}{dt} \int \vert \partial_x V_\ep \vert^2 + \frac{13}{16}\ep \int  \vert  \px^2 V_\ep \vert^2\\
	&\leq 4\left(\lrho_\ep^{-1/2} \ep^{-1} \|\sqrt{\rho_\ep} u_\ep\|_{L^2_x}^2 + \ep^{-3} \lrho_\ep^{-2\alpha} \|V_\ep\|_{L^2_x}^2 \right)  \Vert \px V_\ep \Vert_{L^2_x}^2 \\
	&+ C(\ep,\bar{\rho}_\ep, \lrho_\ep) \Vert \px^2 \re \Vert_{L^2_x}^2 \Vert \px V_\ep \Vert_{L^2_x}^2+ C(\ep,\bar{\rho}_\ep, \lrho_\ep) \left(	\Vert V_\ep \Vert_{L^2_x}^3 	\Vert \px V_\ep \Vert_{L^2_x} + 	\Vert V_\ep \Vert_{L^2_x}^4 \right).
\end{split}
\end{align}
Now we want to derive an expresiion of $ \px^3 \ue $ in terms of $ V_\ep $ and its derivatives. 
Clearly, a direct calculation gives us 
\begin{align*}
	\px^2 V_\ep = 2 \px (\lambda_\ep (\re) ) \px^2 \ue +  \px^2 (\lambda_\ep (\re) ) \px \ue +  \lambda_\ep (\re)  \px^3 \ue .
\end{align*}
Now using relation between $ \px V_\ep $ and $ \px \ue $, we have 
\begin{align*}
	\px^3 \ue &=\dfrac{1}{\lambda_\ep (\re) } \px^2 V_\ep - \dfrac{2\px \lambda_\ep (\re)  }{\lambda_\ep (\re) } \px V_\ep + 2 \left( \dfrac{ \px \lambda_\ep (\re) }{\lambda_\ep (\re) } \right)^2 V_\ep -\left( \dfrac{\lambda^{\prime}_\ep (\re) \vert \px \re \vert^2+ \lambda^{\prime \prime}_\ep (\re) \px^2 \re }{\lambda_\ep (\re)^2 }\right) V_\ep\\
	&:= \sum_{i=1}^{4} B_i.
\end{align*}
Considering $ \tilde{A}=\dfrac{\px \lambda_\ep (\re)}{\lambda_\ep (\re)} $, we observe that \begin{align*}
	\Vert \tilde{A} \Vert_{L^\infty_x}  \leq C(\bar{\rho}_\ep, \lrho_\ep) \Vert \px \re \Vert_{L^\infty_x} \leq C(\bar{\rho}_\ep, \lrho_\ep)  \Vert \px^2 \re \Vert_{L^2_x} .
\end{align*} 
Using the above estimate, we obtain the following bounds: 
\begin{align*}
&	\Vert B_1 \Vert_{L^2_x} \leq \left\Vert \frac{1}{\lambda_\ep (\re) } \right\Vert_{L^\infty_x} 	\Vert \px^2 V_\ep \Vert_{L^2_x} \leq C(\bar{\rho}_\ep, \lrho_\ep) \Vert \px^2 V_\ep \Vert_{L^2_x};\\
&	\Vert B_2 \Vert_{L^2_x} \leq C \Vert \tilde{A} \Vert_{L^\infty_x} \Vert \px V_\ep \Vert_{L^2_x} \leq C(\bar{\rho}_\ep, \lrho_\ep) \Vert \px^2 \re \Vert_{L^2_x}  \Vert \px V_\ep \Vert_{L^2_x} ;\\
&	\Vert B_3 \Vert_{L^2_x} \leq C \Vert \tilde{A} \Vert_{L^\infty_x} \Vert \tilde{A} \Vert_{L^2_x} \Vert V_\ep \Vert_{L^\infty_x} \leq C(\bar{\rho}_\ep, \lrho_\ep) \Vert \px \re \Vert_{L^\infty_x}\Vert \px \re \Vert_{L^2_x}\Vert V_\ep \Vert_{H^1_x} \\
&\quad \quad \quad \; \leq C(\bar{\rho}_\ep, \lrho_\ep) \Vert \px^2 \re \Vert_{L^2_x}  \Vert \px \re \Vert_{L^2_x}\Vert V_\ep \Vert_{H^1_x}  ;\\
&	\Vert B_4 \Vert_{L^2_x} \leq  C(\bar{\rho}_\ep, \lrho_\ep) \Vert \px \re \Vert_{L^\infty_x}\Vert \px \re \Vert_{L^2_x}\Vert V_\ep \Vert_{H^1_x} +  C(\bar{\rho}_\ep, \lrho_\ep) \Vert \px^2 \re \Vert_{L^2_x}\Vert V_\ep \Vert_{H^1_x} \\
&\quad \quad \quad \; \leq C(\bar{\rho}_\ep, \lrho_\ep) \left(\Vert \px^2 \re \Vert_{L^2_x}  \Vert \px \re \Vert_{L^2_x}\Vert V_\ep \Vert_{H^1_x} +  \Vert \px^2 \re \Vert_{L^2_x}\Vert V_\ep \Vert_{H^1_x}  \right) ;
\end{align*}

Therefore, we have 
\begin{align}\label{u3}
	\begin{split}
	\Vert \px^3 \ue \Vert_{L^2_x} \leq &C(\bar{\rho}_\ep, \lrho_\ep) \left(\Vert \px^2 V_\ep \Vert_{L^2_x} +\left( \Vert \px \re \Vert_{L^2_x}\Vert V_\ep \Vert_{H^1_x}  + \Vert V_\ep \Vert_{H^1_x}   \right)\Vert \px^2 \re \Vert_{L^2_x} \right)\\
\end{split}
\end{align}
We recall \begin{align*}
	\frac{1}{2}\frac{d}{dt} \Vert \partial_x^2 \rho_\ep \Vert_{L^2_x}^2 & \leq C\left(    \Vert \px^2 \ue \Vert_{L^2_x}  \Vert \px^2 \re \Vert_{L^2_x}^2 +  \Vert \px^2 \re \Vert_{L^2_x}  \Vert \px^{3} \ue \Vert_{L^2_x}  \right)
\end{align*}
and substitute \eqref{u3} in the estimate for $ \px^2 \re $, we get
\begin{align*}
	\frac{1}{2}\frac{d}{dt} \Vert \partial_x^2 \rho_\ep \Vert_{L^2_x}^2 & \leq C(\bar{\rho}_\ep, \lrho_\ep)\Vert \px^2 V_\ep \Vert_{L^2_x}\Vert \px^2 \re \Vert_{L^2_x}\\
	&+ C(\bar{\rho}_\ep, \lrho_\ep) \left(  \Vert \px^2 \ue \Vert_{L^2_x} +\Vert V_\ep \Vert_{H^1_x} +\Vert \px \re \Vert_{L^2_x}\Vert V_\ep \Vert_{H^1_x}   \right) \Vert \px^2 \re \Vert_{L^2_x}^2 
\end{align*}
Now we add the above estimate with \eqref{V2} to obtain
\begin{align*}
	&\frac{1}{2}\frac{d}{dt} \Vert \partial_x^2 \rho_\ep \Vert_{L^2_x}^2+ \frac{1}{2}\frac{d}{dt}  \Vert \partial_x V_\ep \Vert_{L^2_x}^2 + \frac{13}{16}\ep  \Vert  \px^2 V_\ep \Vert_{L^2_x}^2\\
	& \leq C(\bar{\rho}_\ep, \lrho_\ep)\Vert \px^2 V_\ep \Vert_{L^2_x}\Vert \px^2 \re \Vert_{L^2_x}\\
	&+ C(\bar{\rho}_\ep, \lrho_\ep) \left(  \Vert \px^2 \ue \Vert_{L^2_x} +\Vert V_\ep \Vert_{H^1_x} +\Vert \px \re \Vert_{L^2_x}\Vert V_\ep \Vert_{H^1_x}  +\Vert \px V_\ep \Vert_{L^2_x}^2 \right) \Vert \px^2 \re \Vert_{L^2_x}^2   \\
	& +4\left(\lrho_\ep^{-1/2} \ep^{-1} \|\sqrt{\rho_\ep} u_\ep\|_{L^2_x}^2 + \ep^{-3} \lrho_\ep^{-2\alpha} \|V_\ep\|_{L^2_x}^2 \right)  \Vert \px V_\ep \Vert_{L^2_x}^2 \\
	&+  C(\ep,\bar{\rho}_\ep, \lrho_\ep) \left(	\Vert V_\ep \Vert_{L^2_x}^3 	\Vert \px V_\ep \Vert_{L^2_x} + 	\Vert V_\ep \Vert_{L^2_x}^4 \right).
\end{align*}

At first we use Young's inequality to deduce \[  \Vert \px^2 V_\ep \Vert_{L^2_x}\Vert \px^2 \re \Vert_{L^2_x}\leq \frac{\epsilon}{16} \Vert \px^2 V_\ep \Vert_{L^2_x}^2+\frac{4}{\epsilon}\Vert \px^2 \re \Vert_{L^2_x}^2.\]

Similarly, adjusting a few more terms we have the following inequality:
\begin{align}\label{1final}
	&\frac{1}{2}\frac{d}{dt} \Vert \partial_x^2 \rho_\ep \Vert_{L^2_x}^2+ \frac{1}{2}\frac{d}{dt}  \Vert \partial_x V_\ep \Vert_{L^2_x}^2 + \frac{3}{4}\ep  \Vert  \px^2 V_\ep \Vert_{L^2_x}^2 \leq F_1(t) \Vert \px^2 \re \Vert_{L^2_x}^2 + F_2(t)\Vert \px V_\ep \Vert_{L^2_x}^2+ G(t)
\end{align}
where 
\begin{align*}
&	F_1(t)= C(\ep, \bar{\rho}_\ep, \lrho_\ep) \left(  \Vert \px^2 \ue \Vert_{L^2_x}  +\Vert V_\ep \Vert_{H^1_x} +  \Vert \px \re \Vert_{L^2_x}\Vert V_\ep \Vert_{H^1_x}+\Vert \px V_\ep \Vert_{L^2_x}^2 +1 \right) \\
&	F_2(t)= 4\left(\lrho_\ep^{-1/2} \ep^{-1} \|\sqrt{\rho_\ep} u_\ep\|_{L^2_x}^2 + \ep^{-3} \lrho_\ep^{-2\alpha} \|V_\ep\|_{L^2_x}^2 \right) 
\end{align*}
and 
\begin{align*}
	G(t)= C(\ep,\bar{\rho}_\ep, \lrho_\ep) \left(	\Vert V_\ep \Vert_{L^2_x}^3 	\Vert \px V_\ep \Vert_{L^2_x} + 	\Vert V_\ep \Vert_{L^2_x}^4 \right).
\end{align*}
From our earlier estimates  time interval $ (0,T) $, we have 
\begin{align*}
	&\Vert F_1 \Vert_{L^1_t} \leq C(\ep, \bar{\rho}_\ep, \lrho_\ep,T)  \left(  \Vert \px^2 \ue \Vert_{L^2_t L^2_x}  +\Vert V_\ep \Vert_{L^2_t H^1_x} +  \Vert \px \re \Vert_{L^\infty_x L^2_x}\Vert V_\ep \Vert_{L^2_tH^1_x}+\Vert \px V_\ep \Vert_{L_t^2 L^2_x}^2 +T \right);\\
	&\Vert F_2 \Vert_{L^1_t} \leq C(\ep, \bar{\rho}_\ep, \lrho_\ep,T)  \left(\lrho_\ep^{-1/2} \ep^{-1} \|\sqrt{\rho_\ep} u_\ep\|_{L^\infty_t L^2_x}^2 + \ep^{-3} \lrho_\ep^{-2\alpha} \|V_\ep\|_{L^\infty_t L^2_x}^2 \right);\\
	& \Vert G \Vert_{L^1_t} \leq  C(\ep,\bar{\rho}_\ep, \lrho_\ep,T) \left(	\Vert V_\ep \Vert_{L^\infty_t L^2_x}^3 	\Vert \px V_\ep \Vert_{L^2_t L^2_x} + 	\Vert V_\ep \Vert_{L^\infty_t L^2_x}^4 \right).
\end{align*}
At this point, we observe that for, it holds
\begin{align*}
	\Vert F_1 \Vert_{L^1_t}+	\Vert F_2 \Vert_{L^1_t}+	\Vert G \Vert_{L^1_t} \leq C(\ep,\bar{\rho}_\ep, \lrho_\ep,E_1,T). 
\end{align*}
Integrating the equation \eqref{1final} with respect to time along with the additional hypothesis 
\[  \Vert \px^2 \rho_{0,\ep} \Vert_{L^2_x} +\Vert \partial_x V_{0,\ep} \Vert_{L^2_x} < \infty ,\] 
and Gr\"onwall's inequality, we conclude
\begin{align*}
	&\Vert \px^2 \re \Vert_{L^\infty L^2}+  \| \px V_\ep\|_{L^\infty L^2}^2 + \dfrac{\ep}{2} \|\partial_x^2 V_\ep\|_{L^2 L^2}^2 \\
	& \leq C(\ep, {E_{2}}, \lrho_\ep, T). 
\end{align*}
\emph{$ L^\infty L^2 $ estimate for $ \px^2 \ue $:}
\begin{align*}
	\|\partial^2_x u_\ep\|_{L^\infty L^2}
	& \leq \left\| \dfrac{1}{\lambda_\ep(\rho_\ep)}\right\|_{L^\infty_{t,x}} \Bigg( \|\partial_x V_\ep\|_{L^\infty L^2}
	+ \Vert \lambda_\ep^\prime(\re)\Vert_{L^\infty_{t,x}} (\|\partial_x^2 \re\| _{L^\infty L^2}+ \|\px \re \|_{L^\infty L^2}) \|\partial_x u_\ep\|_{L^\infty L^2} \Bigg)\\
	& \leq C(\ep, {E_{2}}, \lrho_\ep, T). 
\end{align*}
\emph{$ L^2 L^2 $ estimate for $ \px^3 \ue $:}
From the expression
\begin{align*}
	\px^3 \ue &=\dfrac{1}{\lambda_\ep (\re) } \px^2 V_\ep - \dfrac{2\px \lambda_\ep (\re)  }{\lambda_\ep (\re) } \px V_\ep + 2 \left( \dfrac{ \px \lambda_\ep (\re) }{\lambda_\ep (\re) } \right)^2 V_\ep -\left( \dfrac{\lambda^{\prime}_\ep (\re) \vert \px \re \vert^2+ \lambda^{\prime \prime}_\ep (\re) \px^2 \re }{\lambda_\ep (\re)^2 }\right) V_\ep 
\end{align*}
and 
estimate \eqref{u3}, we get
\begin{align*}
		\Vert \px^3 \ue \Vert_{L^2_t L^2_x} \leq &C(\bar{\rho}_\ep, \lrho_\ep) \bigg(\Vert \px^2 V_\ep \Vert_{L^2_t  L^2_x} +\left( \Vert \px \re \Vert_{L^\infty_t L^2_x}^\frac{1}{2}\Vert V_\ep \Vert_{L^\infty_t  H^1_x}    \right)^{\frac{4}{3}} + \Vert \px^2 \re \Vert_{L^2_t L^2_x}^2 + \Vert V_\ep \Vert_{H^1_x} \Vert \px^2 \re \Vert_{L^2_x}  \bigg)\\
	&+C(\bar{\rho}_\ep, \lrho_\ep) (\Vert \px \re \Vert_{L^\infty_t  L^2_x} + \Vert \px \re \Vert_{L^\infty_t L^2_x}^2)\Vert V_\ep \Vert_{L^\infty_t H^1_x} .
\end{align*}
Therefore, we prove our desired Lemma.
\end{proof}

\begin{comment}
		Differtiating continuity equation and multiplying by $\px \re  $, we deduce 
	\[ \frac{1}{2}\frac{d}{dt} \int \vert \partial_x \rho_\ep \vert^2 = -\int \partial_x \rho_\ep \left( \partial_x(u_\ep \partial_x \rho_\ep )+ \partial_x(\rho_\ep \partial_x u_\ep) \right) , \] and eventually, using integration by parts 
	\begin{align*}
		\frac{1}{2}\frac{d}{dt} \int \vert \partial_x \rho_\ep \vert^2 = - \frac{3}{2} \int \px \ue \vert \px \re \vert^2 - \int \re \px^2 \ue \px \re .  
	\end{align*}
	Therefore, we have 
	\begin{align*}
		\frac{1}{2}\frac{d}{dt}  \Vert \partial_x \rho_\ep \Vert^2_{L^2_x}  \leq \Vert \px \ue \Vert_{L^\infty_x} \Vert \px \re \Vert^2_{L^2_x}+ \Vert \re \Vert_{L^\infty_x} \Vert \px^2 \ue \Vert_{L^2_x} \Vert \px \re \Vert_{L^2_x}
	\end{align*}
\end{comment}

\begin{comment}
    \begin{lem}
\begin{align}
& \|\partial^k_x \rho_\ep\|_{L^\infty L^2} + \|\partial_x^{k-1} V_\ep \|_{L^\infty L^2} + \|\partial^k_x V_\ep\|_{L^2 L^2} + \|\partial^k_x u_\ep \|_{L^\infty L^2} + \|\partial^{k+1}_x u_\ep\|_{L^2 L^2} \nonumber \\
& \leq C(\ep, \red{E_{2}(?k?)}, \|\partial^k_x u_{0,\ep}\|_{L^2}, \|\partial^k_x \rho_{0,\ep}\|_{L^2}, \lrho_\ep, T). 
\end{align}
\end{lem}
\end{comment}
\medskip

\begin{proof}[Proof of Proposition \ref{prop-reg-m3}]
In order to prove the proposition, at first we notice that this corresponds to the case $m=3$ and $l=2$ in \eqref{rho-m-1} and \eqref{V-m-1}, respectively.

For $ m=3 $ in \eqref{rho-m-2} we obtain
\begin{align}\label{rho-m-3}
	\begin{split}
		\frac{1}{2}\frac{d}{dt} \Vert \partial_x^3 \rho_\ep \Vert_{L^2_x}^2 & \leq C\left(    \Vert \px^3 \ue \Vert_{L^2_x}  \Vert \px^3 \re \Vert_{L^2_x}^2 +  \Vert \px^3 \re \Vert_{L^2_x}  \Vert \px^{4} \ue \Vert_{L^2_x}  \right)
	\end{split}
\end{align}
Similarly, $ l=2 $ gives us
\begin{align}\label{V-l-2}
	\begin{split}
		\frac{1}{2}\frac{d}{dt} \int \vert \partial_x^2 V_\ep \vert^2 =& - \int \px^2 \left( \left( \ue+  \dfrac{\lambda_\ep(\rho_\ep)}{\rho_\ep^2}\partial_x \rho_\ep\right) \partial_x V_\ep \right) \px^2 V_\ep \\&
		+\int \px^2 \left(  \dfrac{\lambda_\ep(\rho_\ep)}{\rho_\ep} \partial^2_x V_\ep \right) \px^2 V_\ep\\
		& - \int \px^2 \left( \dfrac{\big( \lambda'_\ep(\rho_\ep) \rho_\ep + \lambda_\ep(\rho_\ep)\big)}{\big(\lambda_\ep(\rho_\ep)\big)^2} V_\ep^2 \right) \px^2 V_\ep.
	\end{split}
\end{align}
Application of integration by parts for the terms in the right hand side of the above equation followed by an adjustment of terms leads us to get
\begin{align*}
		&\frac{1}{2}\frac{d}{dt} \int \vert \partial_x^2 V_\ep \vert^2 + \int   \dfrac{\lambda_\ep(\rho_\ep)}{\rho_\ep} \vert  \px^3 V_\ep \vert^2\\
		&= \int \ue \px^2 V_\ep \px^3 V_\ep + \int \px \ue \px V_\ep \px^3 V_\ep +\int \dfrac{\lambda_\ep(\rho_\ep)}{\rho_\ep^2} \px^2 \re \px V_\ep \px^3 V_\ep \\
		&+\int \left(\dfrac{\re \lambda_\ep^\prime (\re)-\lambda_\ep(\re)}{\lambda_\ep^2}\right)\vert \px \re \vert^2 \px V_\ep \px^3 V_\ep  + \int \px \left( \dfrac{\big( \lambda'_\ep(\rho_\ep) \rho_\ep + \lambda_\ep(\rho_\ep)\big)}{\big(\lambda_\ep(\rho_\ep)\big)^2} \right) V_\ep^2 \px^3 V_\ep\\
		&+2 \int \left( \dfrac{\big( \lambda'_\ep(\rho_\ep) \rho_\ep + \lambda_\ep(\rho_\ep)\big)}{\big(\lambda_\ep(\rho_\ep)\big)^2} \right) V_\ep \px V_\ep  \px^3 V_\ep:= \sum_{i=1}^{6} K_i.
\end{align*}

Recalling the inequality $\dfrac{\lambda_\ep(\rho_\ep)}{\rho_\ep} 
\geq \ep $, we conclude \[\int   \dfrac{\lambda_\ep(\rho_\ep)}{\rho_\ep} \vert  \px^3 V_\ep \vert^2 \geq \ep \int    \vert  \px^3 V_\ep \vert^2. \]
\par
\emph{Control for $ K_1 $ :} Here we have
\begin{align*}
	\vert K_1 \vert \leq \Vert \ue \Vert_{L^\infty_x} \Vert \px^2 V_\ep \Vert_{L^2_x} \Vert \px^3 V_\ep \Vert_{L^2_x} \leq \dfrac{\ep}{16}  \Vert \px^3 V_\ep \Vert_{L^2_x}^2 + \frac{4}{\epsilon} \Vert \ue \Vert_{L^\infty_x}^2 \Vert \px^2 V_\ep \Vert_{L^2_x}^2.
\end{align*}
Proceeding similarly as in \eqref{J1}, we obtain
\begin{align}\label{K1}
	\vert K_1 \vert \leq \dfrac{\ep}{16}  \Vert \px^3 V_\ep \Vert_{L^2_x}^2 + 4\left(\lrho_\ep^{-1/2} \ep^{-1} \|\sqrt{\rho_\ep} u_\ep\|_{L^2_x}^2 + \ep^{-3} \lrho_\ep^{-2\alpha} \|V_\ep\|_{L^2_x}^2 \right)  \Vert \px^2 V_\ep \Vert_{L^2_x}^2.
\end{align}
\par
\emph{Control for $ K_2 $ :} Also, for this term we use Young's inequality and the inequality \eqref{embed-1} to get
\begin{align*}
	\vert K_2 \vert \leq \Vert \px \ue \Vert_{L^2_x} \Vert \px V_\ep \Vert_{L^\infty_x} \Vert \px^3 V_\ep \Vert_{L^2_x} \leq \dfrac{\ep}{16}  \Vert \px^3 V_\ep \Vert_{L^2_x}^2 + \frac{8}{\epsilon} \Vert \px \ue \Vert_{L^2_x}^2 \Vert \px^2 V_\ep \Vert_{L^2_x}^2.
\end{align*}
Hence, we have 
\begin{align}\label{K2}
			\vert K_2 \vert  \leq \dfrac{\ep}{16}  \Vert \px^3 V_\ep \Vert_{L^2_x}^2 +C(\ep,\bar{\rho}_\ep, \lrho_\ep) \left( \Vert \px \ue \Vert_{L^2_x}^2 \Vert \px^2 V_\ep \Vert_{L^2_x}^2+ \Vert \px \ue \Vert_{L^2_x}^2 \Vert \px V_\ep \Vert_{L^2_x}^2\right) .
\end{align}
\par
\emph{Control for $ K_3$ :} We note that 
\begin{align}\label{K3}
	\vert K_3 \vert &\leq \left\Vert  \dfrac{\lambda_\ep(\rho_\ep)}{\rho_\ep^2}\right\Vert_{L^\infty_t}  \Vert \px^2 \re \Vert_{L^\infty_x} \Vert \px V_\ep \Vert_{L^2_x} \Vert \px^3 V_\ep \Vert_{L^2_x} \leq \dfrac{\ep}{16} \Vert \px^3 V_\ep \Vert_{L^2_x}^2 + C(\ep,\bar{\rho}_\ep, \lrho_\ep) \Vert \px V_\ep \Vert_{L^2_x}^2  \Vert \px^3 \re \Vert_{L^2_x}^2 .
\end{align}

\par
\emph{Control for $ K_4$ :} We start with the following estimate:
\begin{align*}
		\vert K_4 \vert \leq \left\Vert \dfrac{\re \lambda_\ep^\prime (\re)-\lambda_\ep(\re)}{\lambda_\ep^2}\right\Vert_{L^\infty_x}  \Vert \px \re \Vert_{L^\infty_x}^2  \Vert \px V_\ep \Vert_{L^2_x} \Vert \px^3 V_\ep \Vert_{L^2_x}. 
\end{align*}
Youngs inequality gives us
 \begin{align}\label{K4}
	\vert K_4 \vert \leq \dfrac{\ep}{16} \Vert \px^3 V_\ep \Vert_{L^2_x}^2 +C(\ep,\bar{\rho}_\ep, \lrho_\ep) \Vert \px V_\ep \Vert_{L^2_x}^2 \Vert \px^2 \re \Vert_{L^2_x}^4 .
\end{align}
\par
\emph{Control for $ K_5$ :}
A direct calculation gives us the following identity
\begin{align*}
	\px \left( \dfrac{\big( \lambda'_\ep(\rho_\ep) \rho_\ep + \lambda_\ep(\rho_\ep)\big)}{\big(\lambda_\ep(\rho_\ep)\big)^2} \right) &= \left[\frac{1}{\lambda_\ep(\re)^2} \left( \lambda_\ep^{\prime \prime}(\re) \re + 2\lambda_\ep^{\prime } (\re) \right)- \frac{\lambda_\ep^{ \prime}(\re) }{\lambda_\ep(\re)^3 } \left( \lambda'_\ep(\rho_\ep) \rho_\ep + \lambda_\ep(\rho_\ep)\right)\right] \px \re.
\end{align*}
As a consequence we have 
\begin{align*}
	\left\Vert \left[\frac{1}{\lambda_\ep(\re)^2} \left( \lambda_\ep^{\prime \prime}(\re) \re + 2\lambda_\ep^{\prime } (\re) \right)- \frac{\lambda_\ep^{ \prime}(\re) }{\lambda_\ep(\re)^3 } \left( \lambda'_\ep(\rho_\ep) \rho_\ep + \lambda_\ep(\rho_\ep)\right)\right] \right\Vert_{L^\infty_x} \leq C(\bar{\rho}_\ep, \lrho_\ep).
\end{align*}
hence, we obtain
\begin{align*}
	\vert K_5 \vert \leq C(\bar{\rho}_\ep, \lrho_\ep)  \Vert \px \re \Vert_{L^2_x}  \Vert  V_\ep \Vert_{L^\infty_x}^2 \Vert \px^3 V_\ep \Vert_{L^2_x}. 
\end{align*}
Using the Young inequality we deduce
\begin{align}\label{K5}
		\vert K_5 \vert \leq \dfrac{\ep}{16} \Vert \px^3 V_\ep \Vert_{L^2_x}^2 +  C(\ep,\bar{\rho}_\ep, \lrho_\ep) \Vert \px \re \Vert_{L^2_x}^2  \Vert  V_\ep \Vert_{H^1_x}^4.
\end{align}
\par
\emph{Control for $ K_6 $ :}
Here we observe that
\begin{align*}
	\vert K_6 \vert \leq \left\Vert \left( \dfrac{\big( \lambda'_\ep(\rho_\ep) \rho_\ep + \lambda_\ep(\rho_\ep)\big)}{\big(\lambda_\ep(\rho_\ep)\big)^2} \right)  \right\Vert_{L^\infty_x} \Vert \px V_\ep \Vert_{L^2_x}  \Vert   V_\ep \Vert_{L^\infty_x} \Vert \px^3 V_\ep \Vert_{L^2_x}. 
\end{align*}
This implies 
\begin{align}\label{K6}
		\vert K_6 \vert \leq \dfrac{\ep}{16} \Vert \px^3 V_\ep \Vert_{L^2_x}^2 +  C(\ep,\bar{\rho}_\ep, \lrho_\ep)  \Vert  V_\ep \Vert_{H^1_x}^4.
\end{align}
Therefore, adding inequalities \eqref{K1}-\eqref{K6}, we have 
\begin{align}\label{V-l-2a}
	\begin{split}
		&\frac{1}{2}\frac{d}{dt} \int \vert \partial_x^2 V_\ep \vert^2 + \frac{5}{8}\ep \int   \vert  \px^3 V_\ep \vert^2\\
		&\leq C(\ep,\bar{\rho}_\ep, \lrho_\ep) \Vert \px V_\ep \Vert_{L^2_x}^2 \Vert \px^3 \re \Vert_{L^2_x}^2 \\
		&+ C(\ep,\bar{\rho}_\ep, \lrho_\ep) \left(\left(\lrho_\ep^{-1/2} \ep^{-1} \|\sqrt{\rho_\ep} u_\ep\|_{L^2_x}^2 + \ep^{-3} \lrho_\ep^{-2\alpha} \|V_\ep\|_{L^2_x}^2 \right) +  \Vert \px \ue \Vert_{L^2_x}^2  \right) \Vert \px^2 V_\ep \Vert_{L^2_x}^2 \\
		&+ C(\ep,\bar{\rho}_\ep, \lrho_\ep) \left(\Vert \px V_\ep \Vert_{L^2_x}^2 \Vert \px^2 \re \Vert_{L^2_x}^4 + (1+\Vert \px \re \Vert_{L^2_x}^2 ) \Vert  V_\ep \Vert_{H^1_x}^4 \right)
	\end{split}
\end{align}
Next, we would like to estimate $ \Vert \px^{4} \ue \Vert_{L^2_x}  $. 
A direct computation leads to the following identity: 
\begin{align*}
	\px^4 \ue &= \frac{1}{\lambda_\ep (\re) }  \px^3 V_\ep - \left( \dfrac{ \lambda_\ep^{\prime} (\re)  }{\lambda_\ep (\re)^2} +  \dfrac{2\lambda_\ep^{\prime} (\re)  }{\lambda_\ep (\re) } \right) \px \re \px^2 V_\ep \\
	&-\left( \left(  \dfrac{2\lambda_\ep^{\prime} (\re)  }{\lambda_\ep (\re)} \right)^\prime -\left( \dfrac{2 \lambda_\ep^{\prime}(\re)^2  }{\lambda_\ep(\re)^2}-\dfrac{\lambda_\ep^{\prime} (\re)  }{\lambda_\ep (\re)^2}\right) \right) \vert \px \re \vert^2 \px V_\ep- \left(2\dfrac{\lambda_\ep^{\prime} (\re)  }{\lambda_\ep (\re)}+\dfrac{\lambda_\ep^{\prime} (\re)  }{\lambda_\ep (\re)^2} \right) \px^2 \re \px V_\ep\\
	&+ \left( 2 \left( \dfrac{2 \lambda_\ep^{\prime}(\re)^2  }{\lambda_\ep(\re)^2}-\dfrac{\lambda_\ep^{\prime} (\re)  }{\lambda_\ep (\re)^2}\right) - \left(\dfrac{ \lambda_\ep^\prime(\re)}{\lambda_\ep(\re)^2}\right)^\prime  \right) \px \re \px^2 \re V_\ep\\
	&+ \left( \dfrac{2 \lambda_\ep^{\prime}(\re)^2  }{\lambda_\ep(\re)^2}-\dfrac{\lambda_\ep^{\prime} (\re)  }{\lambda_\ep (\re)^2}\right)^\prime (\px \re)^3 V_\ep- \left(\dfrac{\lambda_\ep^{\prime} (\re)  }{\lambda_\ep (\re)^2} \right) \px^3 \re V_\ep
\end{align*}
We rewrite the abobe expression as
\begin{align*}
		\px^4 \ue &=  \mathcal{R}_1(\lambda_\ep (\re),\lambda_\ep^{\prime} (\re),\lambda_\ep^{\prime \prime} (\re))   \px^3 V_\ep +  \mathcal{R}_2(\lambda_\ep (\re),\lambda_\ep^{\prime} (\re),\lambda_\ep^{\prime \prime} (\re)) \px \re \px^2 V_\ep\\
		&+ \mathcal{R}_3(\lambda_\ep (\re),\lambda_\ep^{\prime} (\re),\lambda_\ep^{\prime \prime} (\re))\vert \px \re \vert^2 \px V_\ep+ \mathcal{R}_4(\lambda_\ep (\re),\lambda_\ep^{\prime} (\re),\lambda_\ep^{\prime \prime} (\re))\px^2 \re \px V_\ep\\
		&+ \mathcal{R}_5(\lambda_\ep (\re),\lambda_\ep^{\prime} (\re),\lambda_\ep^{\prime \prime} (\re))\px \re \px^2 \re V_\ep+ \mathcal{R}_6(\lambda_\ep (\re),\lambda_\ep^{\prime} (\re),\lambda_\ep^{\prime \prime} (\re))(\px \re)^3 V_\ep\\
		&+ \mathcal{R}_7(\lambda_\ep (\re),\lambda_\ep^{\prime} (\re),\lambda_\ep^{\prime \prime} (\re)) \px^3 \re V_\ep := \sum_{i=1}^{7} D_i ,
\end{align*}
where for each $ i=1,\cdots,7 $ we have 
\begin{align*}
	\Vert  \mathcal{R}_i(\lambda_\ep (\re),\lambda_\ep^{\prime} (\re),\lambda_\ep^{\prime \prime}(\re)) \Vert_{L^\infty_x} \leq C(\bar{\rho}_\ep, \lrho_\ep) .
\end{align*}
Therefore, we the following estimates:
\begin{align*}
	&	\Vert D_1 \Vert_{L^2_x}  \leq C(\bar{\rho}_\ep, \lrho_\ep) \Vert \px^3 V_\ep \Vert_{L^2_x};\\
	&	\Vert D_2 \Vert_{L^2_x}  \leq C(\bar{\rho}_\ep, \lrho_\ep) \Vert \px^2 V_\ep \Vert_{L^2_x}\Vert \px \re  \Vert_{L^\infty_x} \leq  C(\bar{\rho}_\ep, \lrho_\ep) \Vert \px^2 V_\ep \Vert_{L^2_x}\Vert \px^2 \re  \Vert_{L^2_x} ;\\
	&	\Vert D_3 \Vert_{L^2_x}  \leq C(\bar{\rho}_\ep, \lrho_\ep) \Vert \px V_\ep \Vert_{L^2_x}\Vert \px \re  \Vert_{L^\infty_x}^2 \leq   C(\bar{\rho}_\ep, \lrho_\ep) \Vert \px V_\ep \Vert_{L^2_x} \Vert \px^2 \re  \Vert_{L^2_x}^2;\\
	&	\Vert D_4 \Vert_{L^2_x}  \leq C(\bar{\rho}_\ep, \lrho_\ep) \Vert \px V_\ep \Vert_{L^2_x}\Vert \px^2 \re  \Vert_{L^\infty_x}\leq  C(\bar{\rho}_\ep, \lrho_\ep) \Vert \px V_\ep \Vert_{L^2_x} \Vert \px^3 \re  \Vert_{L^2_x};\\
	&	\Vert D_5 \Vert_{L^2_x}  \leq C(\bar{\rho}_\ep, \lrho_\ep) \Vert  V_\ep \Vert_{L^\infty_x}\Vert \px \re  \Vert_{L^\infty_x}\Vert \px^2 \re  \Vert_{L^2_x} \leq   C(\bar{\rho}_\ep, \lrho_\ep) \Vert  V_\ep \Vert_{H^1_x} \Vert \px^2 \re  \Vert_{L^2_x}^2;\\
	&	\Vert D_6 \Vert_{L^2_x}  \leq C(\bar{\rho}_\ep, \lrho_\ep) \Vert  V_\ep \Vert_{L^\infty_x}\Vert \px \re  \Vert_{L^\infty_x} \Vert \px \re  \Vert_{L^2_x}^2 \leq   C(\bar{\rho}_\ep, \lrho_\ep) \Vert  V_\ep \Vert_{H^1_x} \Vert \px \re  \Vert_{L^2_x}^3;\\
	&	\Vert D_7 \Vert_{L^2_x}  \leq C(\bar{\rho}_\ep, \lrho_\ep) \Vert  V_\ep \Vert_{L^\infty_x}\Vert \px^3 \re  \Vert_{L^2_x} \leq C(\bar{\rho}_\ep, \lrho_\ep) \Vert  V_\ep \Vert_{H^1_x} \Vert \px^3 \re  \Vert_{L^2_x} ;
\end{align*}

Now, going back to \eqref{rho-m-3} and plugging the above estimate in it, we obtain
\begin{align*}
		\frac{1}{2}\frac{d}{dt} \Vert \partial_x^3 \rho_\ep \Vert_{L^2_x}^2 & \leq C(\bar{\rho}_\ep, \lrho_\ep) \Vert \px^3 V_\ep \Vert_{L^2_x} \Vert \px^3 \re \Vert_{L^2_x}  \\
		&+ C(\bar{\rho}_\ep, \lrho_\ep) \left(  \Vert  V_\ep \Vert_{H^1_x} +\Vert \px^3 \ue \Vert_{L^2_x}   \right) \Vert \px^3 \re \Vert_{L^2_x}^2 \\
		&+ C(\bar{\rho}_\ep, \lrho_\ep) \left( \Vert  \re \Vert_{H^2_x} \Vert \px^3 \re \Vert_{L^2_x} \Vert \px^2 V_\ep \Vert_{L^2_x} \right) \\
		&+ C(\bar{\rho}_\ep, \lrho_\ep) \left( \Vert  V_\ep \Vert_{H^1_x} \left(\Vert  \re \Vert_{H^2_x} +\Vert  \re \Vert_{H^2_x}^2 + \Vert  \re \Vert_{H^2_x}^3\right)   \right)\Vert \px^3 \re \Vert_{L^2_x}.
\end{align*}
Now we add the above inequality with \eqref{V-l-2a} and use the following inequality 
\begin{align*}
	C(\bar{\rho}_\ep, \lrho_\ep) \Vert \px^3 V_\ep \Vert_{L^2_x} \Vert \px^3 \re \Vert_{L^2_x} \leq \frac{\ep}{8} \Vert \px^3 V_\ep \Vert_{L^2_x} ^2 + 	C(\ep,\bar{\rho}_\ep, \lrho_\ep)
	 \Vert \px^3 \re \Vert_{L^2_x}^2
\end{align*}
to deduce 
\begin{align}\label{2final}
	&\frac{1}{2}\frac{d}{dt} \Vert \partial_x^3 \rho_\ep \Vert_{L^2_x}^2+ \frac{1}{2}\frac{d}{dt}  \Vert \partial_x^2 V_\ep \Vert_{L^2_x}^2 + \frac{1}{2}\ep  \Vert  \px^3 V_\ep \Vert_{L^2_x}^2 \leq \tilde{F}_1(t) \Vert \px^3 \re \Vert_{L^2_x}^2 + \tilde{F}_2(t)\Vert \px^2 V_\ep \Vert_{L^2_x}^2+ \tilde{G}(t)
\end{align}
where
\begin{align*}
\tilde{F}_1(t)= C(\ep, \bar{\rho}_\ep, \lrho_\ep) \left(  \Vert  V_\ep \Vert_{H^1_x} +\Vert \px^3 \ue \Vert_{L^2_x}  + \Vert \px \ue \Vert_{L^2_x} + \Vert \partial_x^2 V_\ep \Vert_{L^2_x}^2 + \Vert  V_\ep \Vert_{H^1_x}^2 + \Vert \re \Vert_{H^1_x}^2+1 \right) ,
\end{align*}
\begin{align*}
	\tilde{F}_2(t)= C(\ep, \bar{\rho}_\ep, \lrho_\ep)\left( \left(\lrho_\ep^{-1/2} \ep^{-1} \|\sqrt{\rho_\ep} u_\ep\|_{L^2_x}^2 + \ep^{-3} \lrho_\ep^{-2\alpha} \|V_\ep\|_{L^2_x}^2 \right) +  \Vert \px \ue \Vert_{L^2_x}^2 + \Vert \re \Vert_{H^2_x}^2 \right)
\end{align*}
and
\begin{align*}
	\tilde{G}(t)= C(\ep, \bar{\rho}_\ep, \lrho_\ep)\left( \Vert \px^2 \ue \Vert_{L^2_x}^2 +  \Vert \px^3 \ue \Vert_{L^2_x}^2 + \left( \sum_{k=1}^{6} \Vert \re \Vert_{H^2_x}^k \right) \left(  \sum_{k=1}^{4} \Vert  V_\ep \Vert_{H^1_x}^k  \right)\right). 
\end{align*}
We add a few additional terms in the right hand side to write it in this general form.
We note that, in interval $ (0,T) $
\begin{align*}
	\Vert \tilde{F}_1 \Vert_{L^1_t}+	\Vert \tilde{F}_2 \Vert_{L^1_t}+	\Vert \tilde{G} \Vert_{L^1_t} \leq C(\ep,\bar{\rho}_\ep, \lrho_\ep,E_2,T). 
\end{align*}
Now, we introduce an additional hypothesis 
\[  \Vert \px^3 \rho_{0,\ep} \Vert_{L^2_x} +\Vert \partial_x^2 V_{0,\ep} \Vert_{L^2_x} < \infty .\] 
Again we use Gr\"onwall's inequality to conclude
\begin{align*}
	&\Vert \px^3 \re \Vert_{L^\infty L^2}+  \| \px^2 V_\ep\|_{L^\infty L^2}^2 + \ep \|\partial_x^3 V_\ep\|_{L^2 L^2}^2 \\
	& \leq C(\ep, {E_{3}}, \|\partial^k_x u_{0,\ep}\|_{L^2}, \|\partial^k_x \rho_{0,\ep}\|_{L^2}, \lrho_\ep, T).,
\end{align*}
where 
\[E_{3}= E_2+ \Vert \px^3 \rho_{0,\ep} \Vert_{L^2_x} +\Vert \partial_x^2 V_{0,\ep} \Vert_{L^2_x} . \]
We proceed analogously as in the proof of Lemma \ref{lem-reg-m2} to obtain the $ L^\infty L^2 $ estimate of $ \px^3 \ue $ and the $ L^2 L^2 $ estimate of $ \px^4 \ue $.
\end{proof}

Next we will state and prove a generalized Poincar\'e inequality: 
\begin{prop}\label{GP}
There exists a positive constant $ C $ such that the following inequality holds 
\begin{align}\label{GP-ineq}
    \Vert u \Vert_{L^1_x} \leq C\lr{\Vert \px u \Vert_{L^1_x} + \int_{\mathbb{T}} r \vert u \vert dx },
\end{align}
for any $u \in W^{1,1}(\T)$ and any non-negative function $r$ such that
\begin{align}\label{GP-hyp1}
    0<M_0 \leq \int_{\mathbb{T}} r \ dx < \infty ,\;  \quad r \in L^\infty_x(\T).
\end{align}
\end{prop}
\begin{proof}
We prove the statement by methods of contradiction. Suppose \eqref{GP-ineq} is not true, then there exists a sequence $ \{u_n\}_{n\in \mathbb{N}}$ and $ \{r_n\}_{n\in \mathbb{N}}$ such that
\begin{align*}
    \Vert u_n \Vert_{L^1_x}=1,\; \quad \Vert \px u_n \Vert_{L^1_x} + \int_{\mathbb{T}} r_n \vert u_n \vert \ dx \leq \frac{1}{n}
\end{align*}
and 
\begin{align*}
    r_n \rightarrow r \text{ weakly-* in } L^\infty_x.
\end{align*}
Therefore, we have 
\begin{align*}
    \Vert u_n \Vert_{W^{1,1}_x} \leq 2.
\end{align*}
As a consequence of compact embedding of $ W^{1,1}_x$ in $ L^1_x $, we obtain 
\begin{align*}
    u_n \rightarrow u \text{ strongly in } L^1_x.
\end{align*}
Next, the bound $\Vert \px u_n \Vert_{L^1_x} \leq \frac{1}{n} $ yields 
\begin{align*}
     \px u_n \rightarrow 0 \text{ strongly in } L^1_x.
\end{align*}
The above two statements imply
\begin{align*}
    u_n \rightarrow u \text{ strongly in } W^{1,1}_x \text{ and } \px u =0 \text{ a.e.}.
\end{align*}
Now, note that the weak-* convergence of $ r_n $ in $L^\infty_x$ and strong convergence of $ u_n $ in $ L^1_x $ helps us to deduce
\begin{align*}
    \int_{\T} r\ dx = 0, 
\end{align*}
that contradicts the hypothesis \eqref{GP-hyp1}.
\end{proof}

%%%%%%%%%%%%%%%%%%%%%%
%%%%%%%%%%%%%%%%%%%%%%
\newpage
\section*{Acknowledgments}

C. P. is supported by the SingFlows and CRISIS projects, grants ANR-18-CE40-0027 and ANR-20-CE40-0020-01 of the French National Research Agency (ANR). The work of N.C. and E.Z. is supported by the EPSRC Early Career Fellowship no. EP/V000586/1.

\bigskip
\bibliography{biblio}

\end{document}